\documentclass{article}
\usepackage[T1]{fontenc}
\usepackage[utf8]{inputenc}
\usepackage{geometry}
\usepackage{amsmath}
\usepackage{amsthm}
\usepackage[dvipsnames]{xcolor} 
\usepackage[color=Orange!50!white, textwidth=22mm]{todonotes}
\usepackage[unicode=true,
 bookmarks=false,
 breaklinks=false,pdfborder={0 0 1},backref=section,colorlinks=false]
 {hyperref}
\usepackage{comment}
\makeatletter
\theoremstyle{plain}
\newtheorem{thm}{\protect\theoremname}[section]
\theoremstyle{plain}
\newtheorem{lem}[thm]{\protect\lemmaname}
\theoremstyle{remark}
\newtheorem{claim}[thm]{\protect\claimname}
\theoremstyle{plain}
\newtheorem{cor}[thm]{\protect\corollaryname}
\theoremstyle{plain}

\theoremstyle{plain}
\newtheorem{conjecture}[thm]{\protect\conjecturename}%

\usepackage{amssymb}

\makeatother

\providecommand{\claimname}{Claim}
\providecommand{\corollaryname}{Corollary}
\providecommand{\lemmaname}{Lemma}
\providecommand{\theoremname}{Theorem}
\providecommand{\problemname}{Problem}
\providecommand{\conjecturename}{Conjecture}

\newcommand{\is}{\mathrm{IS}}

\begin{document}
\title{On the clique number of random Cayley graphs and related topics} 
\author{David Conlon\thanks{Department of Mathematics, California Institute of Technology, Pasadena, CA 91125. Email:  {\tt dconlon@caltech.edu.} Research supported by NSF Awards DMS-2054452 and DMS-2348859.} \and Jacob Fox\thanks{Department of Mathematics, Stanford University, Stanford, CA 94305. Email: {\tt jacobfox@stanford.edu}. Research supported by NSF Award DMS-2154129.} \and Huy Tuan Pham\thanks{School of Mathematics, Institute for Advanced Study, Princeton, NJ 08540. Email: {\tt htpham@caltech.edu}. Research supported by a Clay Research Fellowship and a Stanford Science Fellowship.} \and Liana Yepremyan\thanks{Department of Mathematics, Emory University, 
Atlanta, GA 30322. Email: {\tt lyeprem@emory.edu}. Research supported by NSF Award DMS-2247013.}}
\date{}
\maketitle

\begin{abstract}
We prove that a random Cayley graph on a group of order $N$ has clique number $O(\log N \log \log N)$ with high probability. This bound is best possible up to the constant factor for certain groups, including~$\mathbb{F}_2^n$, and improves the longstanding upper bound of $O(\log^2 N)$ due to Alon. Our proof does not make use of the underlying group structure and is purely combinatorial, with the key result being an essentially best possible upper bound for the number of subsets of given order that contain at most a given number of colors in a properly edge-colored complete graph. As a further application of this result, we study a conjecture of Alon stating that every group of order $N$ has a Cayley graph whose clique number and independence number are both $O(\log N)$, proving the conjecture for all abelian groups of order $N$ for almost all $N$. For finite vector spaces of order $N$ with characteristic congruent to $1 \pmod 4$, we prove the existence of a self-complementary Cayley graph on the vector space whose clique number and independence number are both at most $(2+o(1))\log N$. This matches the lower bound for Ramsey numbers coming from random graphs and solves, in a strong form, a problem of Alon and Orlitsky motivated by information theory. 
\end{abstract}

\section{Introduction}

Given a group $G$ and a subset $S$ of $G$ with the property that $s^{-1} \in S$ whenever $s \in S$, the \emph{Cayley graph} $G_S$ is the graph with vertex set $G$ where two vertices $x$ and $y$ are joined by an edge if and only if $xy^{-1} \in S$. For example, for $N \equiv 1 \pmod 4$ prime, the \emph{Paley graph} $P_N$ is the Cayley graph on $\mathbb{Z}_N$ whose generating set $S$ is the set of quadratic residues, so that vertices $x$ and $y$ are joined by an edge if and only if $x-y$ is a quadratic residue.  

Our first concern in this paper will be with estimating the \emph{clique number}, the order of the largest complete subgraph, of random Cayley graphs. For the usual \emph{binomial random graph} $G_{N,1/2}$, the graph with $N$ vertices where each edge is chosen independently with probability $1/2$, it has been known since seminal work of Erd\H{o}s~\cite{Er47} in the 1940s that the clique number is at most $2 \log_2 N$ with high probability and this bound was later shown to be asymptotically tight (see, for example,~\cite{BE, MATULA} and their references). 

Building on work of Green~\cite{G05}, Green and Morris~\cite{GM16} showed that the same upper bound holds with high probability for the clique number of the \emph{uniform random Cayley graph} on $\mathbb{Z}_N$ with $N$ prime, where each set $\{s, -s\}$ is included in the generating set $S$ independently with probability $1/2$. One of our main results says that there is an absolute constant $C$ such that the uniform random Cayley graph on any group $G$ with $N$ elements has clique number at most $C \log N \log \log N$ with high probability, which improves on an earlier bound of $C \log^2 N$ due to Alon~\cite{A95} and, for some groups such as $\mathbb{F}_2^n$, is best possible up to the constant. 

Surprisingly, the proof of this result does not make use of the structure of groups in any meaningful way and applies in a much more general combinatorial setting, to what we call random entangled graphs. The main input, and arguably our real main result, is an essentially best possible upper bound for the number of subsets of order $n$ that contain at most $Kn$ colors in a proper edge-coloring of the complete graph $K_N$. 

To show the power of this result, we also use it to make progress on an important conjecture of Alon~\cite{A95} on the existence of Ramsey Cayley graphs. A graph on $N$ vertices is said to be {\it $C$-Ramsey} if both its clique number and its \emph{independence number}, the clique number of its complement, are at most $C\log N$. This definition is motivated by an old result of Erd\H{o}s and Szekeres~\cite{ES35} saying that there are no $1/2$-Ramsey graphs (see also the recent breakthrough by Campos, Griffiths, Morris and Sahasrabudhe~\cite{CGMS} which improves this $1/2$ to $1/2 + \varepsilon$ for some absolute constant $\varepsilon>0$), implying that Ramsey graphs, if they exist, are essentially best possible. In the other direction, Ramsey graphs do exist, since the result of Erd\H{o}s~\cite{Er47} on the clique number of random graphs implies that almost all graphs on $N$ vertices are $2$-Ramsey. It has also long been conjectured that the family of Paley graphs should be a good source of Ramsey graphs,  though the difficulty in proving this was one of the sparks that initiated the closer study of both the clique number and the Ramsey properties of random Cayley graphs.

With this notation in hand, 
Alon's conjecture states that there is an absolute constant $C$ such that every finite group has a $C$-Ramsey Cayley graph. The results of Green~\cite{G05} and Green and Morris~\cite{GM16} mentioned earlier imply that this conjecture is true over $\mathbb{Z}_N$. There it is sufficient to consider uniform random Cayley graphs, but over other groups such as $\mathbb{F}_2^n$, where the clique number of the random Cayley graph is $\Omega(\log N \log \log N)$ with high probability, this cannot work. Nevertheless, by using our counting result together with some additional ideas, we are able to prove Alon's conjecture for a range of groups, including abelian groups of order $N$ for almost all $N$ and finite vector spaces of characteristic at least $5$. Arguably our strongest result in this direction says that if $G$ is a finite vector space of order $N$ with characteristic congruent to $1 \pmod 4$, then it has a self-complementary Cayley graph whose clique number and independence number are both at most $(2+o(1))\log N$. This solves, in strong form, a problem of Alon and Orlitsky~\cite{AO95} arising from information theory asking whether self-complementary Ramsey Cayley graphs exist.

We will describe our work on Alon's conjecture in detail below, but first we look more closely at our results on the clique number of random Cayley graphs and the more general combinatorial results from which they follow.

\subsection{Clique number of random Cayley graphs} \label{sec:IntroClique}

Given a group $G$, the \emph{random Cayley graph} $G_p$ is obtained by including, for each $s \neq 1$, the set $\{s,s^{-1}\}$ in the generating set $S$ with probability $p$ independently of all other such sets. In particular, $G_{1/2}$ is the \emph{uniform random Cayley graph} on $G$. The first question of concern to us here is that of estimating the clique number of random Cayley graphs, a natural problem with applications and connections in several fields, including discrete geometry, graph Ramsey theory and information theory (see, for instance, \cite{AAAS} and \cite{AO95}).

The main general result on the clique number of uniform random Cayley graphs to date is due to Alon~\cite{A95},  
who showed that there is an absolute constant $C$ such that if $G$ is a group of order $N$, then the uniform random Cayley graph on $G$ has clique number at most $C\log^2 N$ with high probability (see~\cite{AAAS, A13, AO95, CM} for proofs of various special cases and variants of this result). For the particular case of cyclic groups, this result was improved by Green~\cite{G05}, who showed that there is an absolute constant $C$ such that the clique number of the uniform random Cayley graph on $\mathbb{Z}_N$ is at most $C\log N$ with high probability. This result was later improved by Green and Morris~\cite{GM16} for $N$ prime, in which case they showed that the clique number is at most $(2+o(1))\log N$ with high probability, asymptotically matching the bound for the clique number of the uniform random graph on $N$ vertices.

In the other direction, Green~\cite{G05} observed that the uniform random Cayley graph on $\mathbb{F}_2^n$ has clique number $\Theta(\log N \log \log N)$ with high probability (see also~\cite{Mr17} for more precise estimates). The reason for this is that $\mathbb{F}_2^n$ has many subgroups and if all the nonidentity elements of a subgroup $H$ are included in the generating set $S$, then $H$ forms a clique in the Cayley graph. Over $\mathbb{F}_2^n$ and other groups with many subgroups of an appropriate size, such as abelian groups of bounded exponent, this becomes the dominant source of cliques in random Cayley graphs. 

In light of such examples, our first result gives an upper bound for the clique number of uniform random Cayley graphs over a general finite group which is tight up to the value of the absolute constant $C$. Since the result also holds for $G_p$ for any fixed $0 < p < 1$, we will state it in that level of generality.
 
\begin{thm}\label{main}
For every $0 < p < 1$, there exists a constant $C$ such that, for any finite group $G$ of order $N$, the clique number of the random Cayley graph $G_p$ is at most $C \log N \log \log N$ with high probability. 
\end{thm}

We prove Theorem \ref{main} via the first moment method, for which we need to calculate the expected number of cliques of given order $n$ in a random Cayley graph. This in turn leads to the following deterministic question: given a finite group $G$ of order $N$, how many subsets $A$ of $G$ of size $n$ satisfy $|AA^{-1}| \le K|A|$? Note that here $A^{-1}=\{a^{-1}:a \in A\}$ and $AB = \{ab:a \in A,b \in B\}$. One of our main results is the following general upper bound on this count.

 \begin{thm}\label{thm:count-group}
There is an absolute constant $C$ such that, for any finite group $G$ of order $N$, the number of subsets $A$ of size $n$ with $|AA^{-1}| \le Kn$ is at most 
$N^{C( K+\ln n)}(C K)^n$. 
\end{thm}

The bound in this theorem is again tight up to the value of the absolute constant $C$. Indeed, if $G=\mathbb{F}_2^m$, which has order $N=2^m$, the number of subspaces of order $Kn=2^d$ is 
$$\prod_{i=0}^{d-1} \frac{2^m-2^i}{2^d-2^i} \geq 2^{md-d^2},$$
which is at least $N^{c\log(Kn)}$ provided $Kn \leq N^{1-c}$. The number of subsets of size $n$ in each such subspace is $\binom{Kn}{n} \geq K^n$.  As long as $n=d+\omega(1)$, most such subsets span the subspace and so are subsets of only one such subspace. Thus, under these mild assumptions, we get at least $(1-o(1))N^{c\log(Kn)}K^n$ subsets $A \subset \mathbb{F}_2^m$ with $|A|=n$ and $|A-A| \le Kn$. To show that the factor $N^{CK}$ is also necessary, let $G$ be an abelian group which has a subgroup of order $n-K$. This subgroup, together with any $K$ additional elements, gives a subset $A$ of $n$ elements with $|A-A| \leq K(K-1) + 1 + 2K(n-K)\leq 2Kn$. But the number of such subsets is $\binom{N-n+K}{K}$, which is at least $(N/2K)^K$ provided $n \le N/2$. 

Over certain abelian groups, and particularly $\mathbb{Z}_N$, beginning with the work of Green~\cite{G05} and Green and Morris~\cite{GM16} on the clique number of uniform random Cayley graphs, there has been a large amount of work devoted to understanding the bounds for the number of subsets $A$ of $G$ of size $n$ satisfying $|A+A| \le K|A|$ (and the methods generally extend to counting those $A$ of size $n$ satisfying $|A-A|\leq K|A|$, the problem of interest to us here). For example, if $K = o(n/\log^3 N)$, one can obtain very precise results, accurate up to a multiplicative error of $\exp(o(n))$, through a delicate application of the hypergraph container method (see, for instance,~\cite{Cam, CCMMS, CCSW}). Unfortunately, these results are quite specific to the particular abelian groups under consideration and it seems that entirely different approaches are needed for the general case.  
Our approach here, which we now describe in more detail, is to pass to a more general combinatorial framework and to prove results in this setting which imply, and considerably generalize, our stated results for groups.

\subsubsection*{Combinatorial generalizations of Cayley graphs}

Consider an edge-coloring $c$ of the complete graph $K_N$. An {\it entangled graph} with respect to $c$ is a spanning subgraph of $K_N$ obtained from a subset $S$ of the colors 
whose edge set is the union of all edges whose color is in $S$. 
We say that an edge-coloring is {\it $\Delta$-bounded} if each color class has maximum degree at most $\Delta$ and, more informally, \emph{locally bounded} if it is $\Delta$-bounded for some constant $\Delta$. We will also use the standard phrase \emph{proper edge-coloring} as a synonym for a $1$-bounded edge-coloring. Observe that Cayley graphs can be thought of in these terms, since we can color each edge $xy$ with the set $\{xy^{-1}, yx^{-1}\}$ and each vertex is then incident with the same color at most twice. That is, Cayley graphs are $2$-bounded or, if the exponent of the group is $2$, $1$-bounded entangled graphs. 

Given an edge-coloring $c$ of $K_N$ and $0<p<1$, the \emph{random entangled graph} $G_{c}(p)$ is the entangled graph with respect to $c$ formed by including each color in the defining set $S$ with probability $p$ independently of all other colors. Several important random graph models, including Erd\H{o}s--R\'enyi random graphs and random Cayley graphs, are special cases of random entangled graphs coming from locally-bounded edge-colorings. Indeed, the random Cayley graph $G_p$ is $G_c(p)$ with $c$ the $2$-bounded edge-coloring coming from the group $G$, while the Erd\H{o}s--R\'enyi random graph $G(N,p)$ is $G_c(p)$ with $c$ the $1$-bounded edge-coloring of $K_N$ where all edges have a different color. 

Theorem~\ref{main} is therefore a simple corollary of the following result, which gives an upper bound for the clique number of random entangled graphs coming from locally-bounded edge-colorings. 

\begin{thm}\label{main3}
For every natural number $\Delta$ and $0 < p < 1$, there exists a constant $C$ such that, for any $\Delta$-bounded edge-coloring $c$ of $K_N$, the clique number of the random entangled graph $G_{c}(p)$ is at most  $C \log N \log \log N$ with high probability. 
\end{thm}

The assumption that each color class has bounded degree is essential for the bound in Theorem~\ref{main3} to hold. 
For example, consider the edge-coloring $c$ of the complete graph on vertex set $[N]$ where each edge $(i,j)$ with $i<j$ is given color $i$. 
In particular, each color class is sparse, but the coloring is not locally bounded. However, for any subset $S$ of the set of colors, which may be identified with $[N-1]$, the colors in the clique on vertex set $S$ are all from $S$, so the clique number of the random entangled graph $G_c(p)$ will be $(p + o(1))N$ with high probability for $p$ fixed and $N$ tending to infinity. 

As we already indicated in the Cayley graph case, the real result underlying Theorem~\ref{main3} is a counting result, which, in our combinatorial setting, says that in any proper edge-coloring of $K_N$, there are few vertex subsets of a given order that have few colors on their edges. A similar result for locally-bounded edge-colorings follows in a straightforward fashion (see, for instance, Theorem~\ref{complete}) from this result and an application of Vizing's theorem.

\begin{thm}\label{thm:improvedcounting}
There is an absolute constant $C$ such that, for any proper edge-coloring of $K_N$, the number of subsets with $n$ vertices in which at most $Kn$ colors appear is at most $N^{C( K+\ln n)}(C K)^n$. 
\end{thm}

In fact, a weaker version of Theorem \ref{thm:improvedcounting}, saying that the number of subsets with $n$ vertices in which at most $Kn$ colors appear is at most $N^{CK\ln n}(C K)^n$, is already sufficient to prove Theorem \ref{main3}. Since the proof of this weaker bound is considerably simpler, we give it in full in Section~\ref{sec:simple-count} as a warm up to the proof of Theorem \ref{thm:improvedcounting}. 
The main idea in the proof is to show that for each $n$-vertex subset $A$ that contains few colors, there is a spanning tree in $A$ using very few colors. We show the existence of such a spanning tree through a greedy growth procedure. Starting from a vertex $v_1$ of $A$, we select a sequence of colors $c_i$ and let $V_i$ denote the connected component of $A$ containing $v_1$ in the graph whose edges are all those whose color is in $\{c_0, c_1, \dots, c_{i-1}\}$. In each step, we select the color $c_i$ which maximizes $|V_{i+1}|-|V_i|$, which allows us to conclude that $V_s = A$ for some $s=O(K\log n)$.

The proof of the stronger estimate in Theorem \ref{thm:improvedcounting} requires a more sophisticated growth procedure. The target structure for this procedure is no longer a spanning tree on $A$, but rather a tree spanning a large subset of $A$. This is not just a technical device, because, as shown by the following theorem, $\Omega(K \log n)$ colors are sometimes needed to build a spanning tree of $A$.

\begin{thm}\label{spanningcomplete}
There is an absolute constant $C$ such that, for any proper edge-coloring of $K_n$ which uses at most $Kn$ colors, there is a spanning tree using at most $C K\log(n/K)$ colors. Moreover, provided $K \leq n/(\log n)^8$, this bound is tight up to the constant factor.
\end{thm} 

To build a tree on a large subset of $A$ using very few colors,  
we expose a random sequence of colors and keep track of the evolution of the connected components of the graph on $A$ whose edge set is the union of these colors. For our purposes, the evolution of these components will split into two phases. In the first phase, where we expose $O(K)$ colors, the components grow steadily until a typical component is adjacent to $\Omega(Kn)$ different colors, a constant fraction of all the colors in $A$. Then, in the second phase, the components coalesce, the total number of components decreasing by a multiplicative factor for each new exposed color. Therefore, after $O(\log n)$ more steps, we will have connected a component of order $\Omega(|A|)$, as required. The actual counting result then follows from carefully tracking the number of different colored trees that can arise from this growth procedure together with some additional ideas.

\subsubsection*{Connections to additive dimension}

Given a sequence of elements $S = (c_1,\dots,c_t)$ in a group $G$, define $\mathrm{span}(S)$ to be the set of all elements of the form $\prod_{i=1}^{t} c_i^{w_i}$ for $w_i \in \{-1,0,1\}$. Given a subset $A$ of $G$, we define the \emph{additive (or multiplicative) dimension} $d_*(A)$ of $A$ to be the minimum size of a sequence $S$ such that $A$ is contained in $\mathrm{span}(S)$. There have been numerous previous results regarding the additive dimension of subsets with small doubling in an abelian group $G$. Notably, Sanders \cite{San} (see also Schoen and Shkredov \cite{SS}) showed that, for $G$ abelian, if $A$ is a subset of $G$ of size $n$ with doubling at most $K$, that is, with $|A+A| \leq K|A|$, then $d_*(A) = O(K\log n)$. 

There are other notions of additive/multiplicative dimension which are closely related to $d_*(A)$ in the abelian case. 
A set $S$ in a group $G$ is said to be \emph{dissociated} if the identity element is not equal to any product of the form $\prod_{i=1}^{t} c_i^{w_i}$ with $c_1,\dots,c_t$ a nonempty sequence of distinct elements of $S$ and $w_i \in \{-1,1\}$ for $1 \leq i \leq t$. We may then define $\dim(A)$ to be the maximum size of a dissociated subset $S$ of $A$. In an abelian group, any maximal dissociated subset $S$ of $A$ satisfies $A\subseteq \mathrm{span}(S)$ and, hence, $\dim(A) \ge d_*(A)$. Moreover, Schoen and Shkredov~\cite{SS} observed that $\dim(A) = O(d_*(A)\log d_*(A))$ in abelian groups. The results of \cite{San} and \cite{SS} actually say that if $A$ is a subset of an abelian group $G$ with doubling at most $K$, then $\dim(A) = O(K\log n)$, which then implies the result highlighted in the previous paragraph. In recent work, Alon, Buci\'c, Sauermann, Zakharov and Zamir \cite{ABSZZ} proved a bound for general groups $G$, showing that if $A$ is a subset of $G$ of size $n > 2$ with $|AA|\le K|A|$, then $\dim(A) = O(K\log n\log \log n)$.

Here we show that the upper bound $d_*(A) = O(K\log n)$ extends from abelian groups to general groups and, moreover, that an analogous result holds in the following broader combinatorial setting. 
Consider a proper edge-coloring of the complete graph $K_n$. For a vertex $v$ of this graph and a sequence $S$ of colors, let $\mathrm{span}(S; v)$ be the set of vertices which can be reached from $v$ via a path whose edge colors form a subsequence of $S$. Our combinatorial generalization of the bound $d_*(A) = O(K\log n)$, which is an easy corollary of the proof of our weak version of Theorem~\ref{thm:improvedcounting}, is then as follows.

\begin{cor}\label{cor:gen}
    For any proper edge-coloring of $K_n$ using at most $Kn$ colors and any vertex $v$, there is a sequence $S$ of $O(K\log n)$ colors such that all vertices are contained in $\mathrm{span}(S; v)$. 
\end{cor}

The promised upper bound on $d_*(A)$ for sets $A$ with small doubling in general groups is in turn a simple corollary of this result. 

\begin{cor}\label{cor:dim-weak}
    For any subset $A$ of a group $G$ of size $n$ with $|AA^{-1}| \le K|A|$, there is a set $S$ of size $O(K\log n)$ such that $A\subseteq \mathrm{span}(S)$. That is, $d_*(A) = O(K\log n)$. 
\end{cor}

One might notice that Corollary \ref{cor:dim-weak} considers sets $A$ with $|AA^{-1}| \le K|A|$ rather than $|AA|\le K|A|$. However, the analogous result does also hold for $AA$. The proof requires some minor modifications, mainly, that instead of coloring the edges of the complete graph on $A$, we must color the edges of the complete bipartite graph whose parts are copies of $A$. We say more about this after Theorem~\ref{expand}.

If we make use of ideas from the proof of the full strength version of Theorem \ref{thm:improvedcounting}, we may, in the case of abelian groups, improve the upper bound on $d_*(A)$ in Corollary~\ref{cor:dim-weak} to $d_*(A) = O(K\log \log n+ \log n)$. Moreover, if the group has exponent at most $r$, we may push the bound to $d_*(A) = O(K \log r + \log n)$. These bounds get much closer to the lower bound of $\Omega(K+\log n)$ that comes from considering the union of $K$ generic parallel subspaces of size $n/K$ in $\mathbb{F}_2^m$.    

\begin{thm}\label{thm:dim-strong}
    If $G$ is an abelian group and $A$ is a subset of $G$ of size $n$ with $|A-A| \le K|A|$, then $d_*(A) = O(K\log \log n+ \log n)$. Furthermore, if $G$ has exponent $r$, then $d_*(A) = O(K \log r + \log n)$. 
\end{thm}

As was the case for Corollary~\ref{cor:dim-weak}, a result analogous to Theorem~\ref{thm:dim-strong} also holds for sets $A$ with $|A+A| \le K|A|$. This again follows from working with respect to the complete bipartite graph whose parts are copies of $A$, though we will not go into this in detail. 

In subsequent work~\cite{FP}, using as input some of the combinatorial ideas behind the proof of Theorem~\ref{thm:improvedcounting}, the second and third authors establish a conjecture of Ruzsa which sharpens the above upper bound on $d_*(A)$ for abelian groups with bounded exponent. More precisely, Ruzsa~\cite{R99} conjectured that, in any abelian group $G$ with exponent $r$, any subset $A$ of $G$ of size $n$ with $|A-A|\le K|A|$ must be contained in a subgroup of size at most $\exp(O_r(K))n$. The result in \cite{FP} confirms this conjecture, showing that any such $A$ must be contained in a subgroup of size at most $r^{(2+o_K(1))K}n$, which is sharp up to the $o_K(1)$ term. Note that this implies and gives a stronger conclusion than the upper bound $d_*(A) \le O(K\log r+\log n)$, which would only yield that $A$ is contained in a subgroup of size at most $r^{O(K \log r + \log n)} = r^{O(K \log r)}n^{O(\log r)}$. 

\subsection{Alon's conjecture on Ramsey Cayley graphs} \label{sec:IntroAlon}

Our other main results, for which Theorem~\ref{thm:count-group} is an essential input, concern the following striking conjecture of Alon~\cite{A95} (see also~\cite{A02,AOpen}).

\begin{conjecture}[Alon]
There is a constant $C$ such that every finite group has a $C$-Ramsey Cayley graph.
\end{conjecture}

While this conjecture was verified by Green~\cite{G05} in the particular case of cyclic groups by studying the clique number of the uniform Cayley graph, essentially nothing was previously known in the general case. Indeed, as we have seen above, even when $G$ is an abelian group, it is not always the case that a random Cayley graph on $G$ is Ramsey with high probability, a fact which led Green~\cite{G05} to remark that `one should not expect the full conjecture, if it is true, to be easy to prove'. Despite this difficulty, by analyzing the clique number of appropriate non-uniform random Cayley graphs, we are able to prove Alon's conjecture for a broad range of groups. For instance, we have the following result.

\begin{thm}\label{main2}
There is a constant $C$ such that if the largest factor of $N$ which is a product of powers of $2$ and $3$ is at most $(\log N)^{1/1000}$, then, for every abelian group $G$ of order $N$, there is a $C$-Ramsey Cayley graph on $G$. 
\end{thm}

In particular, this theorem applies when $G = \mathbb{F}_5^n$, so there are groups for which the uniform random Cayley graph fails to be Ramsey with high probability, but for which Ramsey Cayley graphs nevertheless exist. One takeaway from Theorem~\ref{main2} which we highlight is that Alon's conjecture holds for abelian groups of order $N$ for almost all $N$.

\begin{cor}\label{cor3}
There is a constant $C$ such that, for almost all $N$, every abelian group $G$ of order $N$ has a $C$-Ramsey Cayley graph on $G$. 
\end{cor}

For finite vector spaces of characteristic at least $5$, we can prove a much more precise result, saying that there are Cayley graphs whose clique and independence number both asymptotically match those in the binomial random graph $G_{N, 1/2}$. 

\begin{thm}\label{thm:Fp}
For every vector space $V$ of order $N$ whose characteristic is at least $5$, there is a Cayley graph on $V$ whose clique and independence numbers are both at most $(2+o(1))\log_2 N$.
\end{thm}

Our methods also allow us to solve a longstanding problem of Alon and Orlitsky~\cite{AO95} asking whether there are Ramsey Cayley graphs which are self-complementary, where a graph is \emph{self-complementary} if it is isomorphic to its complement. Without the extra stipulation that the graph be Cayley, self-complementary graphs whose clique and independence numbers are both at most $(2+o(1))\log_2 N$ were constructed by McDiarmid and Steger \cite{MS}. We prove a similar result in the Cayley setting.

\begin{thm}\label{thm:Fq}
For every vector space $V$ of order $N$ whose characteristic is congruent to $1 \pmod 4$, there is a self-complementary Cayley graph on $V$ whose clique and independence numbers are both at most $(2+o(1))\log_2 N$.
\end{thm}

The motivation for the problem of Alon and Orlitsky came from information theory, where they studied how much one can save in zero-error communication problems through repeated use of a channel.  
Over a channel with confusion graph $G$, the maximum possible number of messages that can be sent without error with $k$ uses of the channel is the independence number $\alpha(G^k)$ of $G^k$, the strong product of $k$ copies of $G$. Normalized, the number of messages per use of the channel over $k$ uses is $c_k(G):=\alpha(G^k)^{1/k}$. The \emph{Shannon capacity} of $G$ is the limit $c(G):=\lim_{k \to \infty} c_k(G)$, which exists as the independence number of the strong product is supermultiplicative. If $G$ is a self-complementary Ramsey graph on $N$ vertices, which Alon and Orlitsky~\cite{AO95} show exist for $N$ a multiple of $4$, then $\alpha(G)=O(\log N)$ and $\alpha(G^2) \geq N$. Such a graph $G$ satisfies $c_1(G)=O(\log N)$ and $c_2(G) \geq \sqrt{N}$. Thus, one can get substantially more efficient communication by repeated use of such a channel rather than through a single use. 

In their work, Alon and Orlitsky were interested in whether a variant of their result holds for \emph{dual-source coding}. 
They observed that such a variant, which brings in the chromatic number rather than the independence number of strong graph powers, would follow from the existence of self-complementary Ramsey Cayley graphs, which is now addressed by Theorem \ref{thm:Fq}. To say more, we next explain the dual-source coding problem. 

Suppose that there are two sets $X$ and $Y$ and a collection of pairs $S\subseteq X\times Y$. A sender is then given an element $x$ of $X$ and a receiver an element $y$ of $Y$ such that $(x, y) \in S$. Assume that both the receiver and the sender know $S$ and their own positions in $X$ and $Y$. What is the smallest number of messages the sender needs available to them to allow them to communicate the value of $x$? If we let $G$ be the \emph{characteristic graph} with vertex set $X$ where $x$ and $x'$ are adjacent if and only if there exists $y$ with $(x,y),(x',y)\in S$, then it is a simple exercise to show that the answer to this question is $\chi(G)$, where $\chi(G)$ is the chromatic number of $G$. We also note that, for any given graph $G$, there exist $X$, $Y$ and $S$ for which $G$ is the associated graph.

Suppose now that the sender is given $k$ elements $x_1, \dots, x_k \in X$ and the receiver $k$ elements $y_1,\dots,y_k \in Y$ such that $(x_i,y_i)\in S$ for all $i$. What is the smallest number of messages the sender now needs available to them to allow them to communicate the values of $x_1, \dots, x_k$? The answer is $\chi(G^k)$. The reason for Alon and Orlitsky's question is their observation that if $G$ is a self-complementary Cayley graph with $N$ vertices, then $\chi(G^{2}) \le N$. Therefore, if there is a self-complementary Cayley graph $G$ with $N$ vertices whose independence number is $O(\log N)$ and $G$ is the characteristic graph of a dual source, then at least $\chi(G) = \Omega(N/\log N)$ messages must be available to communicate one element of $X$, but at most $N$ messages are needed to communicate two elements of $X$. Our Theorem~\ref{thm:Fq} answers Alon and Orlitsky's question in a strong form, asymptotically matching the lower bound for Ramsey numbers coming from random graphs. It also improves on work of Alon and Orlitsky themselves, who showed that there are self-complementary Cayley graphs with $N$ vertices whose independence number is $O(\log^2 N)$.

Getting back to the proofs of Theorems \ref{main2}, \ref{thm:Fp} and \ref{thm:Fq}, we have already mentioned that, because uniform random Cayley graphs can fail to be Ramsey for finite vector spaces, we must instead consider a different distribution on random Cayley graphs. This distribution is designed so that subspaces cannot form cliques or independent sets and, more generally, that any potential clique or independent set $A$ must have large doubling $|A-A|/|A|$. Combined with our counting result, Theorem \ref{thm:count-group}, this expansion property of potential cliques allows us to remove the $\log\log |G|$ factor from the bound on the clique number. Unlike the results of Section~\ref{sec:IntroClique}, we do need to make use of additive structure here. Indeed, our method relies on showing that any potential clique lacks solutions to certain systems of linear equations and we then make use of known results in additive combinatorics to prove the required expansion property.

The more precise results in Theorems \ref{thm:Fp} and \ref{thm:Fq} also require an additional input, namely, an essentially tight upper bound on the Freiman dimension of sets with small doubling in finite vector spaces. For $|A+A|/|A|$ instead of $|A-A|/|A|$, such a result was already shown by Even-Zohar and Lovett \cite{EZL}. We get the result we need by appropriately modifying their technique.

\paragraph{Organization of the paper.} In Section \ref{sec:simple-count}, we give the short proof of a weak version of Theorem~\ref{thm:improvedcounting} that is already sufficient to recover Theorems \ref{main} and \ref{main3}. In Section \ref{sec:count-main}, we then give the full details of the proofs of our main counting results, Theorems \ref{thm:count-group} and \ref{thm:improvedcounting}. We prove Theorem~\ref{spanningcomplete} on spanning trees with few colors in Subsection~\ref{subsectionconnectedspanning} and Theorem~\ref{thm:dim-strong} on the additive dimension of sets with small doubling in Section~\ref{adddimsubsect}. In Section \ref{sec:Alon}, we prove our results on Alon's conjecture, Theorems \ref{main2}, \ref{thm:Fp} and \ref{thm:Fq}. We conclude with some further remarks on open problems and ongoing related work.

\section{Counting subsets with few colors} \label{sec:simple-count}

In this section, we give a first estimate for the number of vertex subsets of given order in a locally-bounded edge-colored complete graph that contain few colors inside. This estimate will already be sufficient to show that the clique number of random entangled graphs with $N$ vertices, and therefore random Cayley graphs over groups of order $N$, is $O(\log N \log \log N)$ with high probability. We begin by working under the assumption that the edge-coloring is proper, that is, that each color class has maximum degree at most one. We will then use Vizing's theorem to lift this result to locally-bounded edge-colorings.

\begin{thm}\label{proper}
There is an absolute constant $C$ such that in any properly edge-colored complete graph on $N$ vertices, the number of subsets of $n$ vertices in which at most $Kn$ colors appear is at most $N^{C K\ln n}(C K)^n$. 

\end{thm}

\begin{proof} 

Label the vertex set of the complete graph as $[N]=\{1,\ldots,N\}$. 
Instead of directly counting subsets, we will give an upper bound on the number of sequences of $n$ distinct vertices which have at most $Kn$ colors between them. Each set of $n$ vertices gives rise to $n!>(n/e)^n$ such sequences, so it will suffice to prove  that there is an absolute constant $C$  such that the number of sequences of $n$ distinct vertices which have at most $r=Kn$ colors between them is at most $N^{C K\ln n}(Cr/e)^n$. 

Let $V=(v_i)_{i=1}^n$ denote an arbitrary (but not yet determined) sequence of $n$ distinct vertices that induce at most $r$ colors between them. Observe that $r \geq n-1$ as the induced coloring on $V$ is also proper. We view $V$ as a set together with an ordering of its vertices. We can fix $v_1$ as any one of the $N$ possible vertices and set $V_0=\{v_1\}$ and $n_0=1$. After step $i$, we will have picked out a subset $V_i$ of $n_i$ distinct vertices together with the location of each of these vertices in the ordering. Furthermore, we will be able to determine both $V_i$ and the location of its elements using only a small number of choices.   These subsets will be nested, in that $V_{i-1}$ is a subset of $V_i$ for each $i$ for which $V_i$ is defined. We stop after step $i$ if $n_i=n$ (so that $V_i=V)$.  Otherwise,  $n_i<n$ and we continue on to step $i+1$.  We then let $u_{i}=\lceil n_i(n-n_i)/r \rceil$ and $n_{i+1}=n_i+u_i$. In particular, the sequence $(n_i)$ only depends on $n$ and $r$. 

Observe that between $V_i$ and $V \setminus V_i$ there are $n_i(n-n_i)$ edges and, as at most $r$ colors are used on these edges, there is a color $c_i$ with at least $n_i(n-n_i)/r$ edges of color $c_i$ between $V_i$ and $V \setminus V_i$.  Furthermore, there are at most $\binom{N}{2} < N^2$ choices for the color $c_i$. Since each color forms a matching there are $u_i$ vertices in $V \setminus V_i$ that form a matching to $V_i$ in the color $c_i$. We can pick the indices (in the sequence $V$) of the vertices of this matching which are in $V \setminus V_i$ in $\binom{n-n_i}{u_i}$ ways. Suppose these indices are $j_1,\ldots,j_{u_i}$. 
We can pick which vertices in $V_i$ they each connect to in the matching in at most $(n_i)_{u_i} \leq n_i^{u_i}$ ways. That is, we have at most $n_i$ choices for the vertex with index $j_1$ to match to in $V_i$, then at most $n_i-1$ choices for the vertex with index $j_2$ to match to in $V_i$ and so on. As the $u_i$ vertices in $V_i$ that are matched to the vertices in $V \setminus V_i$ are already determined from these choices and the color $c_i$ forms a matching, the $u_i$ vertices in $V \setminus V_i$ with indices  $j_1,\ldots,j_{u_i}$ are determined as the unique vertices that are adjacent in color $c_i$ to the corresponding vertices in $V_i$. 
We add these $u_i$ vertices in $V \setminus V_i$ to $V_i$ to obtain $V_{i+1}$. 
In total,  we made at most $N^2 \binom{n-n_i}{u_i}n_i^{u_i}$ choices to determine the vertices in $V_{i+1} \setminus V_i$ together with their location in the ordering provided we had already determined $V_i$ and the location of its elements in the ordering. 

Note that if $n_i<n$, then $u_i \leq n-n_i$ is a positive integer and so $n_s=n$ and $V_s=V$ for some positive integer $s$. Moreover, if $n_i \leq n/2$, then $$n_{i+1}=n_i+u_i \geq n_i\left(1+(n-n_i)/r\right) \geq n_i(1+1/2K) \geq n_i e^{1/4K}.$$ 
Therefore, after at most $4K\ln n$ steps, we arrive at $n_i \geq n/2$. If $n_i \geq n/2$, then $$n-n_{i+1}=n-n_i-u_i \leq n-n_{i}-(n-n_i)/2K=(n-n_{i})(1-1/2K) \leq (n-n_i)e^{-1/2K}.$$ 
Therefore, after another at most $2K\ln n$ additional steps, we stop with $n=n_s$. Hence, $s \leq  6 K \ln n $.  

Thus, there are at most $$N\prod_{i=0}^{s-1} N^2 \binom{n-n_i}{u_i}n_i^{u_i}
\leq N^{2s+1} \prod_{i=0}^{s-1}\left(e\frac{n-n_i}{u_i}\right)^{u_i}n_i^{u_i} \leq N^{2s+1}(er)^{n-1} \leq N^{12K \ln n +1}(er)^n $$ choices for the sequence $V$. The second inequality follows from substituting in the value of $u_i$ and using the identity $\sum_{i=0}^{s-1} u_i = n-1$. Indeed, using induction one can show that  $\sum_{i=0}^{j-1} u_i$ is the number of edges in a tree on $n_j$ vertices and $n=n_s$. The last inequality follows from the upper bound on $s$. This completes the proof. 
 \end{proof}

Let us quickly note that the proof above also yields Corollary \ref{cor:gen}. Recall the statement, that, in any proper edge-coloring of $K_n$ using at most $Kn$ colors, there is a vertex $v$ and a sequence $S$ of $O(K\log n)$ colors such that all vertices are contained in $\mathrm{span}(S; v)$.

\begin{proof}[Proof of Corollary \ref{cor:gen}]
    Following the proof of Theorem \ref{proper}, we can find a sequence of colors $c_1,\dots,c_s$ with $s\le 6K\ln n$ and sets of vertices $\{v_1\} = V_0 \subseteq V_1 \subseteq \dots \subseteq V_s = V(K_n)$ such that all vertices of $V_{i+1}\setminus V_i$ have an edge with color $c_i$ to $V_i$. It follows immediately from induction that for all vertices $v$ in $V_s$ there is a path from $v_1$ such that the sequence of colors on the path is a subsequence of the sequence $c_1,\dots,c_s$. 
\end{proof}

We also deduce Corollary \ref{cor:dim-weak}, which states that if $A$ is a subset of a group $G$ of size $n$ with $|AA^{-1}| \le K|A|$, then $d_*(A) = O(K\log n)$. That is, there is a set $S$ of size $O(K\log n)$ such that $A\subseteq \mathrm{span}(S)$. 

\begin{proof}[Proof of Corollary \ref{cor:dim-weak}]
    Given a group $G$ and $A\subseteq G$ with $|AA^{-1}|\le K|A|$, consider the edge-coloring of the complete graph on $A$ using $r \le K|A|$ colors where $\{x,y\}$ receives the color $\{xy^{-1},yx^{-1}\}$. Note that each color class is either $1$-regular or $2$-regular. Let $n = |A|$. By partitioning the color classes that are $2$-regular, we obtain a proper edge-coloring of the complete graph on $A$ using at most $3Kn$ colors. 

    Applying Corollary \ref{cor:gen} to this proper edge-coloring, we obtain a sequence of colors $c_1,\dots,c_s$ with $s = O(K \log n)$ and a vertex $v_1$ such that every vertex is connected to $v_1$ via a path whose colors form a subsequence of $c_1,\dots,c_s$. Each color $c_i$ corresponds to an element of $G$ (more precisely, each color corresponds to an equivalence class $\{g,g^{-1}\}$ from which we make an arbitrary choice of an element $g\in G$). It is then easy to check that $A\subseteq \left\{v_1\prod_{i=1}^{s}c_i^{w_i}: w_i\in \{-1,0,1\}\right\}$. This implies $d_*(A) = O(K\log n)$, as required. 
\end{proof}

As promised, we now lift Theorem~\ref{proper} to locally-bounded edge-colorings using Vizing's theorem.

\begin{thm} \label{complete}
There is an absolute constant $C$ such that in any edge-colored complete graph on $N$ vertices in which each color class has maximum degree at most $\Delta$, the number of subsets of $n$ vertices in which at most $Kn$ colors appear is at most $N^{C\Delta K\ln n}(C\Delta K)^n$. 
\end{thm}

\begin{proof} 
Vizing's theorem \cite{V64} says that every graph of maximum degree $\Delta$ can be properly edge-colored with $\Delta+1$ colors. 
We may therefore split each color class in our edge-coloring into at most $\Delta+1$ color classes so that we get a proper edge-coloring. Any subset which had at most $Kn$ colors originally now has at most $(\Delta+1)Kn$ colors. By Theorem \ref{proper}, the new proper edge-coloring of $K_N$ has at most $N^{C (\Delta+1)K\ln n}(C (\Delta+1)K)^n$ subsets with at most $(\Delta+1)Kn$ colors. Thus, the original edge-coloring of $K_N$ has at most  $N^{C (\Delta+1)K\ln n}(C (\Delta+1)K)^n$ subsets with at most $Kn$ colors. Substituting in $\Delta +1 \leq 2\Delta$ and increasing $C$ by a factor of $2$, we get the desired result. 
\end{proof}

The same proof actually gives a more general result. We call a graph $G=(V,E)$ a {\it $\lambda$-edge-expander} if, for every subset $S$ of $V$, 
$e(S,V\setminus S) \geq \lambda |S|(n-|S|)$. Note that a complete graph is a $1$-edge-expander and a balanced complete bipartite graph is a $1/2$-edge-expander.

\begin{thm} \label{expand}
There is an absolute constant $C$ such that in any edge-colored graph on $N$ vertices in which each color class has maximum degree at most $\Delta$, the number of subsets $S$ of $n$ vertices in which there are $Kn$ colors such that these $Kn$ colors induce a $\lambda$-edge-expander on $S$ is at most 
$N^{C\Delta \lambda^{-1}K\ln n}(C\Delta \lambda^{-1}K)^n$. 
\end{thm}

We may use this result to prove the variant of Corollary~\ref{cor:dim-weak} with $AA$ instead of $AA^{-1}$. Indeed, since $K_{n,n}$ is a $1/2$-edge-expander, we can apply the same argument to that used in the proof of Corollary~\ref{cor:gen} to conclude that for any proper edge-coloring of $K_{n,n}$ using at most $Kn$ colors, there is a vertex $v$ and a sequence $S$ of $O(K\log n)$ colors such that all vertices are contained in $\mathrm{span}(S; v)$. The promised variant of Corollary~\ref{cor:dim-weak} then follows by considering the coloring of the edges of the complete bipartite graph whose parts are copies of $A$ where the color of each edge $(a,b)$ is equal to $ab$. The key observation is that if $v$ is the starting vertex, chosen to be on the second side of the bipartition, and $w$ is another vertex on the same side that can be reached via a path of edges with colors $c_1, \dots, c_t$, then $w = v c_1^{-1} c_2 c_3^{-1} \cdots c_t$.

The following corollary of Theorem \ref{complete} comes from simply coloring the edge $(a,b)$ of the complete graph on $G$ with one of the colors $ab$ or $ba$ arbitrarily or, alternatively, from coloring $(a,b)$ with the color $\{ab, ba\}$. 

\begin{cor}\label{countgroup}
There is a constant $C$ such that if $G$ is a finite group of order $N$, then the number of subsets $A \subset G$ with $|A|=n$ and $|AA| \le K|A|$ is at most $N^{CK\ln n}(CK)^{n}$. Similarly, the number of subsets $A \subset G$ with $|A|=n$ and $|AA^{-1}| \le K|A|$ is at most $N^{CK\ln n}(CK)^{n}$.
\end{cor}

Given a group $G$ of order $N$, consider the complete balanced bipartite graph with $2N$ vertices whose parts $V_1$ and $V_2$ are each a copy of $G$ and where each edge $(a,b) \in V_1 \times V_2$ is colored $ab$.  Each color class in this coloring forms a perfect matching and so has maximum degree one. Therefore, using the fact that the balanced complete bipartite graph is a $\frac{1}{2}$-edge-expander, we have the following corollary of Theorem \ref{expand}. 

\begin{cor}\label{contgroup2}
There is a constant $C$ such that if $G$ is a finite group of order $N$, 
then the number of pairs of subsets $A,B \subset G$ with $|A|=|B|=n$ and $|AB| \le K|A|$ is at most $N^{CK\ln n}(CK)^{2n}$.
\end{cor}

We now turn our attention to estimating the clique number of random entangled graphs. It will be slightly more convenient to phrase our results in terms of estimating the independence number rather than the clique number, but the two problems are clearly equivalent. We first note a probabilistic lemma.

\begin{lem}\label{simp}
Let $G=(V,E)$ be a graph on $n$ vertices with maximum degree $\Delta$ and let $t$ be a positive integer with $\Delta t \leq n/8$.  If $T$ is a uniform random subset of $t$ vertices from $G$ and $X$ is the random variable counting the number of edges in the induced subgraph $G[T]$, then the probability that $T$ is not an independent set is at least $\min(1/4,\mathbb{E}[X]/2)$.   
  \end{lem}

Note that the lower bound on the probability in Lemma \ref{simp} 
is tight up to a constant factor since the probability is trivially at most $1$ and, as $X$ is a nonnegative integer-valued random variable, at most $\mathbb{E}[X]$. 

\begin{proof}[Proof of Lemma \ref{simp}]
The event that $T$ is not an independent set is the same as the event that $X \geq 1$.  We may assume that $\textrm{Pr}[X \geq 1] < 1/4$, as otherwise we are done.  Under this assumption, we prove that $\textrm{Pr}[X = 1] \geq \mathbb{E}[X]/2$, which will complete the proof.  We have \begin{equation}\label{eq12}\textrm{Pr}[X = 1]=\sum_{e \in E}\textrm{Pr}[e \subset T] \, \textrm{Pr}[f \not \subset T \, \, \forall \, f\in E \setminus \{e\} \, | \, e \subset T].\end{equation}
For incident edges $e$ and $f$,  the probability that $f \subset T$ conditioned on $e \subset T$ is $\frac{t-2}{n-2}$.  As each edge $e \in E$ is incident to at most $2(\Delta-1)$ other edges,  the probability that there is an edge of $G$ incident to $e$ that is in $T$ conditioned on $e \subset T$ is at most $2(\Delta-1)\frac{t-2}{n-2} \leq 1/4$. Moreover, since the probability a random subset of order $t-2$ contains an edge is at most the probability that a random subset of order $t$ contains an edge, the probability that there is an edge of $G$ not incident to $e$ that is in $T$ conditioned on $e \subset T$ is at most $\textrm{Pr}[X \geq 1]$  and, hence, 
$\textrm{Pr}[\exists f\in E \setminus \{e\} \, \textrm{satisfying} \, f \subset T | e\subset T] \leq 1/4+\textrm{Pr}[X \geq 1] < 1/2$. Substituting into (\ref{eq12}),  we have 
$$\textrm{Pr}[X = 1]\geq \frac{1}{2} \sum_{e \in E}\textrm{Pr}[e \subset T]=\mathbb{E}[X]/2,$$
completing the proof. 
\end{proof} 

We now use this result to prove that locally-bounded edge-colorings of the complete graph contain subsets of any large order with many colors.

\begin{lem}\label{subsetmanycolors}
For any edge-colored complete graph on $n$ vertices in which each color class has maximum degree at most $\Delta \leq n/100$ and any $t \geq \sqrt{n/\Delta}$, there is a subset of $t$ vertices that contains at least $\frac{n}{20\Delta}$ colors. 
\end{lem}

\begin{proof}
Let $t_0$ be the maximum positive integer for which $\Delta {t_0 \choose 2} \leq (n-1)/4$.  Such a $t_0$ satisfies $\frac 12 \sqrt{n/\Delta} \leq t_0 \leq \sqrt{n/\Delta}$. Since $t_0 \geq 5$, we also have $\Delta t_0 \leq n/8$.  We will prove that there is a subset $T$ of $t_0$ vertices that contains at least $\frac{1}{2}{t_0 \choose 2} \geq \frac{n}{20\Delta}$ colors.  As every superset of $T$ contains at least as many colors as $T$ does,  this will complete the proof. 

Consider a random subset $T$ of $t_0$ vertices.  We will show that the expected number of colors in $T$ is at least $\frac{1}{2}{t_0 \choose 2}$ and, hence, that there exists a subset of $t_0$ vertices with at least $\frac{1}{2}{t_0 \choose 2}$ colors inside. 
For any given color $c$ in the edge-coloring, the maximum degree condition implies that the number $m$ of edges in  that color satisfies $m \leq \Delta n/2$. The number $X$ of edges of that color in a uniform random subset of order $t_0$ satisfies 
$$\mathbb{E}[X] = m{t_0 \choose 2}/{n \choose 2} \leq \frac{\Delta n}{2} {t_0 \choose 2}/{n \choose 2} = \Delta {t_0 \choose 2}/(n-1) \leq 1/4.$$ 
Therefore, by Lemma \ref{simp}, which we can apply since $\Delta t_0 \leq n/8$, the probability that $T$ does not form an independent set in color $c$ is at least $\mathbb{E}[X]/2$. Summing over all colors, the expected number of colors in the random subset of $t_0$ vertices is at least $\frac{1}{2}{t_0 \choose 2}$. 
\end{proof}

Note that the bound in Lemma \ref{subsetmanycolors} on the number of colors in the subset is sharp up to a constant factor, as, if $n-1$ is a multiple of $\Delta$ and $n$ is even, we can edge-color the complete graph on $n$ vertices so that each color is $\Delta$-regular and, hence, has exactly $(n-1)/\Delta$ colors.  The lower bound on the subset size is also sharp up to a constant factor,  as a set of $t$ vertices can have at most ${t \choose 2}$ colors inside. 

We now note the main corollary of Theorem~\ref{complete} that will be needed in estimating the independence number of a random entangled graph.

\begin{lem}\label{specialcliques}
Let $c$ be an edge-coloring of $K_N$ in which each color class has maximum degree at most $\Delta$ and consider the random entangled graph $G_c(p)$ in which each color class from this edge-coloring is included with probability $p=1-q$ independently of the other color classes. Let $C$ be the constant from Theorem \ref{complete} and suppose $K_0 \geq 10q^{-1}\ln (C\Delta/q)$ and $t \geq x\ln x$, where $x=10Cq^{-1}\Delta \ln N$. Then, with high  probability, no subset of $t$ vertices with at least $K_0t$ colors is a clique in $G_c(p)$. 
\end{lem}

\begin{proof}
Let $r=Kt$ with $K \geq K_0$. By Theorem \ref{complete}, the number of subsets of $t$ vertices with exactly $r$ colors is at most $N^{C\Delta K\log t}(C\Delta K)^t$. The probability that a given set of $t$ vertices with $r$ colors is a clique in $G_c(p)$ is at most $p^{r}\leq e^{-qr}$. Hence, the expected number of vertex subsets of order $t$ with $r$ colors inside which form a clique is at most $N^{C\Delta K\log t}(C\Delta K)^te^{-qr} \leq e^{-qr/3}$. Indeed, the assumed lower bound on $t$ implies that $N^{C\Delta K\log t} \leq e^{qr/3}$ and the lower bound on $K$ implies that $(C\Delta K)^t \leq e^{qr/3}$. Summing over all $r \geq K_0t$, the probability that there is a subset of $t$ vertices with at least $K_0t$ colors that forms a clique in $G_c(p)$ is at most $$\sum_{r \geq K_0t}e^{-qr/3} \leq e^{-qK_0t/3}(1-e^{-q/3})^{-1}=o(1),$$
as required.
\end{proof}

By combining Lemmas~\ref{subsetmanycolors} and~\ref{specialcliques}, we can now prove the promised result, saying that, for $q>0$ and $\Delta$ fixed and any edge-coloring $c$ of the complete graph $K_N$ in which each color class has maximum degree $\Delta$, the independence number of the random entangled graph $G_c(q)$ is $O(\log N\log \log N)$ with high probability. In particular, this implies Theorems~\ref{main} and \ref{main3}. 

\begin{thm}\label{intermediate}
Let $2 \leq \Delta \leq N-1$, $0<q<1$ and let $c$ be an edge-coloring of $K_N$ in which each color class has maximum degree at most $\Delta$. Let $n=C'q^{-2}\Delta^2\ln (\Delta/q)\ln N \ln(q^{-1}\Delta\ln N)$, where $C'$ is a sufficiently large constant. Then, with high  probability, the independence number of the random entangled graph $G_c(q)$ is less than $n$. 
\end{thm}

\begin{proof}
Let $K = 10q^{-1}\ln (C\Delta/q)$, where $C$ is the constant from Lemma \ref{specialcliques}, and let $t=n/(20\Delta K)$. Note that an independent set in $G_c(q)$ is the same as a clique in $G_c(p)$ with $p=1-q$. By Lemma \ref{specialcliques} and our choice of $n$, which implies that $t \geq x \log x$ with $x=10Cq^{-1}\Delta \ln N$, with high probability no set of $t$ vertices which contains at least $Kt$ colors is an independent set in $G_c(q)$. 

Note that our choice of $n$ guarantees that $t \geq \sqrt{n/\Delta}$ and $\Delta \leq n/100$. Hence, by Lemma \ref{subsetmanycolors}, every set of $n$ vertices in our edge-colored complete graph $K_N$ contains a subset of $t$ vertices with at least $\frac{n}{20\Delta}=Kt$ colors inside.  Thus, with high probability, the independence number of $G_c(q)$ is less than $n$. 
\end{proof}

In the next section, among other things, we will prove Corollary~\ref{Deltaboundedcor}, 
which replaces the $K\ln n$ factor in the exponent of $N$ in Theorem \ref{complete} by $K+\ln n$. In turn, writing $s = \Delta/q$, this implies that we can improve Lemma~\ref{specialcliques} to say that if $K_0 = \Theta(q^{-1}\ln s)$ and $t = \Omega(s \ln N \max(1, q \ln(\frac{\Delta \ln N}{\ln s})/\ln s))$, 
then, with high probability, no subset of $t$ vertices with at least $K_0t$ colors is a clique in $G_c(p)$. Substituting this into the proof of Theorem~\ref{intermediate}, we see that if $n=O(s^2\ln N \max(\ln s, q\ln\left(\frac{\Delta \ln N}{\ln s}\right) ))$, then, with high  probability, the independence number of the random entangled graph $G_c(q)$ is less than $n$. For comparison, Theorem~\ref{intermediate} allows us to take $n = O(s^2 \ln N \ln s \ln(s \ln N))$, so we get a noticeably better bound for small $q$. 

\section{Connected components with few colors and improved counting} \label{sec:count-main}

Let $A$ be a subset of $n$ vertices in a properly edge-colored complete graph with at most $Kn$ colors appearing in $A$. In this section, we study a random process where we select a random subset of the colors and connect vertices in $A$ using these colors. In Section~\ref{sec:complog}, we will describe this process in detail and warm up by showing that exposing $O(K\log K)$ random colors results in a subgraph where most vertices are in connected components whose incident edges to the rest of the graph use at least a constant fraction of the total number of colors. Then, in Section~\ref{sec:compk}, we give an improved, although significantly more difficult, analysis to show that the same conclusion holds with only $O(K)$ colors, which is best possible up to the value of the implied constant. We use this result in Section~\ref{sec:impcount} to improve the results of the last section counting the number of vertex subsets of order $n$  with at most $Kn$ colors inside in a properly edge-colored complete graph. In Section~\ref{subsectionconnectedspanning}, we use our results to estimate the minimum number of colors needed to find a connected spanning subgraph in $A$ and give a construction to show that our bound is tight. Finally, in Section~\ref{adddimsubsect}, we give an improved bound on additive dimension in abelian groups. 

\subsection{Connected components using $O(K\log K)$ random colors} \label{sec:complog}

Before describing our random process, we note a relatively simple lemma showing that we may always assume that each color appears at most $n/K$ times.

\begin{lem}\label{lem:preproc}
Given a vertex set $A$ of order $n$ and a proper edge-coloring of any graph on $A$ with at most $Kn$ colors, there is a proper edge-coloring of the same graph which refines the given coloring, uses at most $2Kn$ colors and where the number of edges in each color class is at most $n/K$. 
\end{lem}

\begin{proof}
We refine the coloring as follows: 
for each color class, we arbitrarily partition the set of edges of that color into sets of size $\lfloor n/K \rfloor$ and a remaining nonempty set (if it exists) of  size smaller than $\lfloor n/K \rfloor$. The number of colors in the new coloring with size less than $\lfloor n/K \rfloor$ is at most the number of colors in the original coloring. On the other hand, the number of colors with $\lfloor n/K \rfloor$ edges is at most ${n\choose 2}/\lfloor n/K \rfloor < Kn$. Hence, the number of colors in the new coloring is at most $2Kn$. 
\end{proof}

We now describe our random process for selecting the colors and fix some notation that we will use in analyzing it. We choose a sequence $s =(s_i)_{i \geq 1}$ of random colors, where each $s_i$ is chosen uniformly at random, allowing repetitions, from the set of colors that appear on $A$. Let $s_{[i]}=\{s_1,\dots,s_i\}$ denote the set consisting of the first $i$ colors in this sequence and let $\mathcal{C}_i$ be the collection of connected components induced by the union of the edges with colors in $s_{[i]}$. For each component $C\in \mathcal{C}_i$, let $N(C)$ denote the set of colors between $C$ and the remaining components. 
For any vertex $v \in A$, we write $N(v)$ for the set of colors of edges incident to $v$ and $C_i(v)$ for the component in $\mathcal{C}_i$ containing $v$. 
For a component $C$ and a vertex subset $U$, we let $N(C; U)$ denote the set of colors that appear on edges between $C$ and $U$. Finally, we let $N^+(\varepsilon; C)=\min_{|U|\ge (1-\varepsilon)n}|N(C; U)|$, where we note that in this subsection we will take $\varepsilon=1/4$. This may be thought of as a robust measure of the number of colors emerging from $C$.

\begin{lem} \label{lem:KlogK}
Consider a proper edge-coloring of the complete graph on a set $A$ of $n$ vertices using at most $2Kn$ colors, where each color class contains at most $n/K$ edges. Then there exists a set $\hat{S}$ of $80K\ln K$ colors such that, for at least $9n/10$ vertices $v \in A$, $v$ is contained in a connected component $C(v)$ formed using edges with colors in $\hat{S}$ satisfying $|N^+(1/4; C(v))| \ge Kn/64$.
\end{lem}

\begin{proof}
For each vertex $v$, we define a (nondecreasing) set of vertices $T_i(v)$ and a set of colors $S_i(v)$ inductively in $i$. For $i=0$, we set $T_0(v)=\{v\}$. In each step $i\ge 0$, we set $S_i(v) = \bigcup_{x\in T_i(v)}N(x)$. We say that a vertex $u$ is $(v,i)$-\emph{bad} if $|N(u)\cap S_i(v)| \ge n/2$ and $(v,i)$-\emph{good} otherwise. We call a vertex $u\in C_i(v)\setminus C_{i-1}(v)$ a $(v,i)$-\emph{core vertex} if $u$ is $(v,i-1)$-good. If $(v,i)$-core vertices exist, we pick an arbitrary $(v,i)$-core vertex $u$ and update $T_i(v)=T_{i-1}(v)\cup \{u\}$. Otherwise, we set $T_i(v)=T_{i-1}(v)$ and proceed to the next step.  

For $w\ge 0$, let $F_w(v)$ be the smallest $i$ such that $|T_i(v)|\ge w+1$. Note that $F_0(v)=0$ as $|T_0(v)|=1$. Let $t_w := F_w(v)-F_{w-1}(v)$. Our aim now is to bound $\mathbb{E}[t_w]$. 

Let $i=F_{w-1}(v)$ and $C=C_{i}(v)$.  Let $u_1,\dots,u_{w}$ be the vertices in $T_i(v)$. Then, by definition, each vertex $u_j$ adds at least $n/2$ new colors $N(u_j)\setminus \bigcup_{j'<j}N(u_{j'})$. Among them, at most $4|S_i(v)|/K$ colors go to vertices that are $(v,i)$-bad. This is because the number of vertices which are $(v,i)$-bad is at most $4|S_i(v)|/K$, by the following claim. 
\begin{claim}\label{claim:easy}
    For any set of colors $S$, the number of $u$ with  $|N(u)\cap S|\ge n/2$ is at most $4|S|/K$.
\end{claim}
\begin{proof}
    The number of edges with some color in $S$ is at most $|S|n/K$. The number of pairs $(u,s)$ where $u\in A$, $s\in S$ and $s\in N(u)$ is thus at most $2|S|n/K$. Hence, the number of $u$ with $|N(u)\cap S|\ge n/2$ is at most $\frac{2|S|n/K}{n/2}=4|S|/K$.  
\end{proof}

Let $Q$ be the set of colors between $C$ and $(v,i)$-good vertices. If $w \le K/16$, then  
\[
|Q| \ge \left(n/2-4|S_i(v)|/K\right)|T_i(v)| \ge wn/4, 
\]
where we note that $|S_i(v)|\le n|T_i(v)|=nw \le Kn/16$. Let $j_Q$ denote the smallest $j \ge i$ such that $s_j \in Q$. 
Since the $C_j(v)$ are nondecreasing and $S_j(v) = S_i(v)$ for all $j \in [i,i+t_w)$, we have that $i+t_w \le j_Q$. But then, since $t_w \le j_Q-i$ and $j_Q-i$ is the number of time steps starting from $i$ until a color in $Q$ is sampled, we have
\[
\mathbb{E}[t_w] \le \frac{2Kn}{|Q|} \le \frac{2Kn}{wn/4} = 8K/w. 
\]

In particular, the expected value of $F_{K/16}(v)$ is at most $$\sum_{w=1}^{K/16} \mathbb{E}[t_w] \leq \sum_{w=1}^{K/16} 8K/w \leq 8K\ln K.$$
Thus, by Markov's inequality, the event $F_{K/16}(v)\ge 80K\ln K$ has probability at most $1/10$. Hence, the expected number of vertices $v$ where we need $i \ge 80K\ln K$ to have $|T_i(v)|\ge K/16$ is at most $n/10$. 

By averaging, there is $i \leq 80K\ln K$ and a sequence of $i$ colors such that at least $9n/10$ vertices $v$ satisfy $|T_i(v)|\ge K/16$. Each $u\in T_i(v)$ contributes at least $n/2$ many distinct colors on the edges between $C_i(v)$ and the entire vertex set. Since $N^+(\varepsilon; C_i(v))$  is defined by taking the minimum of $|N(C_i(v);U)|$ over all vertex sets $U$ with $|U|\geq (1-\varepsilon)n$ and the coloring is proper, we may lose at most $\varepsilon n$ colors for each $u \in T_i(v)$. Hence, for all $v$ with $|T_i(v)|\ge K/16$, we have 
\[
|N^+(\varepsilon; C_i(v))| \ge \sum_{j=1}^{|T_i(v)|} (n/2- \varepsilon n)\ge Kn/64,
\] 
where we used that $\varepsilon=1/4$.
\end{proof}

\subsection{Connected components using $O(K)$ random colors} \label{sec:compk}

We now improve the main result of the previous subsection by removing the $\log K$ factor from the number of colors required. This will then allow us to find a large component (of order at least $n/4$) using $O(K+\ln n)$ colors. 

\begin{lem}\label{lem:ball}
Consider a proper edge-coloring of the complete graph on a set $A$ of $n$ vertices using at most $2Kn$ colors, where each color class contains at most $n/K$ edges. Then there exists a set $\hat{S}$ of $900K$ colors such that, for at least $9n/10$ vertices $v \in A$,  $v$ is contained in a connected component $C(v)$ formed using edges with colors in $\hat{S}$ satisfying $|N^+(3/8; C(v))| \ge Kn/128$.
\end{lem}

We first give an overview of our proof strategy and introduce several key parameters and objects. We are given a proper edge-coloring of the complete graph on a set $A$ of $n$ vertices using at most $2Kn$ colors, where each color class contains at most $n/K$ edges.  As before, we pick a sequence of colors $s_1, s_2, \dots$ randomly one at a time. Let $\delta = 1/16$ and $W=\log (\delta K)$. Given the (infinite) sequence $s = (s_i)_{i\ge 1}$ of random colors, we consider the following procedure which defines, for each vertex $v$ and each integer $j\in [1,2^W]$, a set  of vertices $B_j = B_j(v)$ of order $j$ which is determined by the sequence of colors $s$. 
Note that here and throughout the remainder of this section, we fix a vertex $v$ and drop the dependence on $v$ in the notation. 

We let $B_1 = \{v\}$. For each $j\ge 2$, we define $P_j$ to be the set of vertices $u$ with $|N(u) \setminus N(B_{j-1})| \ge n/2$, where, for $B \subseteq A$, $N(B)$ denotes the set of colors on the edges between $B$ and $A\setminus B$. 
We then define $t_j$ to be the smallest positive integer distinct from $t_2,\dots,t_{j-1}$ such that the sampled color $s_{t_j}$ connects $B_{2^{\lfloor \log_2(j-1)\rfloor}}$ with a vertex $u$ in $P_j$ (if there is more than one possible choice for $u$, we choose the first in some predetermined ordering of the vertices).  We let $B_j = B_{j-1}\cup \{u\}$. We will run this procedure for $j\le 2^W = \delta K$, setting $T_{j} = \max_{j'\le j}t_{j'}$ at each step.

Our key lemma is the following.

\begin{lem}\label{lem:ball-key}
    Suppose $\kappa \ge 900K$ and $J=2^W=\delta K$. Then
    $$\mathbb{P}(T_{J} > \kappa) \le 2\exp(-\kappa/(128K)).$$
\end{lem}

We now quickly deduce Lemma \ref{lem:ball} from Lemma \ref{lem:ball-key}. 

\begin{proof}[Proof of Lemma \ref{lem:ball}]
    For each vertex $v$ and $B_J = B_J(v)$, $|N^+(3/8; B_J)| \ge J(n/2-3n/8)=Jn/8 = Kn/128$, since each vertex in $B_J$ introduces at least $n/2$ new colors. Moreover, $B_J$ is connected using the colors $s_1,\dots,s_{T_J}$. Therefore, by Lemma \ref{lem:ball-key} with $\kappa = 900K$, the expected number of vertices $v$ for which the component $C$ containing $v$ built using colors $s_1,\dots,s_{\kappa}$ satisfies $|N^+(3/8;C)| \ge Kn/128$ is at most $2n\exp(-\kappa/(128K))$. Hence, there exists a choice of colors $s_1,\dots,s_{\kappa}$ such that at least $9n/10$ vertices $v$ are contained in a component $C$ with $|N^+(3/8;C)| \ge Kn/128$ using the colors $s_1,\dots,s_{\kappa}$. 
\end{proof}

We next turn to the proof of Lemma \ref{lem:ball-key}. 

\begin{proof}[Proof of Lemma \ref{lem:ball-key}]
    For each $j\in (2^w,2^{w+1}]$, define $S_{j} = N(B_{2^w}, P_{j})\setminus N(B_{2^{w-1}})$, where, for $B, P \subseteq A$, $N(B,P)$ denotes the set of colors on the edges between $B$ and $P$. Note that the $S_j$ form a nonincreasing sequence of sets for $j\in (2^w,2^{w+1}]$. For $j\le J$, we have that \begin{equation}\label{eq:S}
|S_j| \ge 2^{w-1}(1/2-4\delta)n,
\end{equation}
since each vertex $B_{j'}\setminus B_{j'-1}$ introduces at least $n/2$ new colors for $j'\in (2^{w-1},2^w]$, at most $4jn/K\le 4\delta n$ of which are not to $P_j$.  
In turn, the latter claim follows since, by Claim \ref{claim:easy}, the number of vertices not contained in $P_j$ is at most $4|B_j|n/K = 4jn/K$. 

Let $H_w$ denote the set $\{t_j: 2^w<j\le 2^{w+1}\}$ of size $2^w$. Note that, by definition, all of the $t_j$ are distinct. Let $I$ be the set of $w\in [0,W-1]$ for which $T_{2^{w+1}} > T_{2^w}$. For each $\tilde{I}\subseteq [0,W-1]$, each sequence $(\tau_{2^{w+1}})_{w\in [0,W-1]}$ of positive integers and each choice of sets $\tilde{H}_w \subseteq [1,\tau_{2^{w+1}}]$ of size $2^w$ for each $w\in \tilde{I}$, we bound the probability of the event $\mathcal{E}$ that $I=\tilde{I}$ and, for all $w \le W-1$, $T_{2^{w+1}}=\tau_{2^{w+1}}$ and $H_w=\tilde{H}_w$. Let $\mathcal{E}_w$ be the event that $T_{2^{w'+1}}=\tau_{2^{w'+1}}$ and $H_{w'}=\tilde{H}_{w'}$ for all $w'\le w$. 
Let $\mathcal{F}_w$ denote the $\sigma$-algebra generated by the $(t_j,s_{t_j})$ for $j\le 2^{w+1}$. Note that $\mathcal{E}_w$ is measurable in $\mathcal{F}_w$. We have the following claim. 

\medskip

\noindent {\bf Claim.} Conditional on $\mathcal{F}_w$, the distribution of the sequence of colors $(s_t)$ can be described as follows. For $t \in \bigcup_{w'\le w} {H}_{w'}$, $s_t$ is determined by $\mathcal{F}_w$. For $t \notin \bigcup_{w'\le w} {H}_{w'}$, $s_t$ is a uniformly random color outside $E_{w,t} = \bigcup_{j\le 2^{w+1}:t_{j}>t} N(B_{2^{\lfloor \log_2(j-1)\rfloor}}, P_j)$. Furthermore, the $s_t$ are conditionally independent given $\mathcal{F}_w$.

\begin{proof}
Consider a sequence of distinct integers $t^*_j$ and colors $s^*_j$ such that there is a positive probability that $t_j = t^*_j$ and $s_{t_j} = s^*_j$ for all $j\le 2^{w+1}$. In particular, there exists a sequence of colors $\tilde{s}_t$ such that the corresponding $t_j(\tilde{s}) = t^*_j$ and $\tilde{s}_{t^*_j} = s^*_j$. For clarity, in this proof, we emphasize the dependence on the sequence of colors ($s$ or $\tilde{s}$) of $t_j, B_j, P_j$. 

For each sequence of colors $(s_t)$ such that $s_{t^*_j} = s^*_j$ for $j\le 2^{w+1}$ and $s_t \notin E_{w,t}$ for all $t\notin \bigcup_{w'\le w}{H}_{w'}$, we will prove by induction on $j$ that $t_j(s) = t_j(\tilde{s})$. As $t_2$ is the smallest time 
at which a color in $N(B_1,P_2)$ is sampled, when $s_t \notin N(B_1,P_2)$ for all $t < t^*_2$ and $s_{t^*_2} = \tilde{s}_{t^*_2}$, 
we have $t_2(s) = t_2(\tilde{s})$. 

For the induction step, assume $t_{j'}(s)=t_{j'}(\tilde{s})$ for all $j'<j$. We have that $B_{j'}(s) = B_{j'}(\tilde{s})$ for all $j'<j$ and, therefore, $P_{j'}(s) = P_{j'}(\tilde{s})$ for all $j'\le j$. Recall that $t_{j}$ is defined as the smallest time distinct from $t_2,\dots,t_{j-1}$ for which a color is sampled that connects $B_{2^{\lfloor \log_2(j-1)\rfloor}}$ to a vertex $u \in P_j$. 
As $B_{2^{\lfloor \log_2(j-1)\rfloor}}(s) = B_{2^{\lfloor \log_2(j-1)\rfloor}}(\tilde{s})$ and $P_j(s)=P_j(\tilde{s})$, it is easy to check that under the conditions $s_{t} \notin N(B_{2^{\lfloor \log_2(j-1)\rfloor}}(\tilde{s}),P_j(\tilde{s}))$ for $t<t^*_j$ distinct from $t^*_2,\dots,t^*_{j-1}$ and $s_{t^*_j}=\tilde{s}_{t^*_j}$, we have $t_j(s)=t_j(\tilde{s})$. 

If a sequence $s$ does not satisfy that $s_t\notin E_{w,t}$ for all $t\notin  \bigcup_{w'\le w} {H}_{w'}$, then there is a smallest $j$ such that $s_t \in N(B_{2^{\lfloor \log_2(j-1)\rfloor}}(\tilde{s}),P_j(\tilde{s}))$ for some $t<t^*_j$ with $t\notin \bigcup_{w'\le w} {H}_{w'}$. We then have inductively that $B_{j'}(s)=B_{j'}(\tilde{s})$ for all $j'<j$ and $P_j(s)=P_j(\tilde{s})$, from which it follows that $t_j(s) < t_j(\tilde{s})$. 

Therefore, conditional on $t_j = t_j^*$, $s_j = s_j^*$ for all $j\le 2^{w+1}$, all consistent sequences of colors $s$ are exactly given by the conditions that $s_{t^*_j} = s^*_j$ for $j\le 2^{w+1}$ and $s_t \notin E_{w,t}$ for all $t\notin \bigcup_{w'\le w}H_{w'}$. This immediately implies the conclusion of the claim. 
\end{proof}

Conditional on $\mathcal{F}_{w-1}$, under the event $\mathcal{E}_w$, we have that 
$s_t \notin S_{2^{w+1}}$ for all $t\in (0,\tau_{2^{w+1}}) \setminus \bigcup_{w'\le w}\tilde{H}_{w'}$ and $s_t \in N(B_{2^w})$ for $t\in \tilde{H}_w$. From the claim, these events are independent for different $t$. Noting that $S_{2^{w+1}}$ is disjoint from $E_{w-1,t} \subseteq N(B_{2^{w-1}})$ and using (\ref{eq:S}), we have that the probability $s_t\notin S_{2^{w+1}}$ conditional on $\mathcal{F}_{w-1}$ is at most $1-\frac{2^{w-1}(1/2-4\delta)n}{2Kn}$. Similarly, as $|N(B_{2^w})|\le 2^wn$, the probability that $s_t\in N(B_{2^w})$ conditional on $\mathcal{F}_{w-1}$ is at most $\frac{2^wn}{Kn/4} = \frac{2^{w+2}}{K}$, which we used that the number of colors is at least $\binom{n}{2}/(n/K) \geq Kn/4$. Thus, 
\begin{align*}
\mathbb{P}(\mathcal{E}_w \mid \mathcal{F}_{w-1})
&\le 
\left(\frac{2^{w+2}}{K}\right)^{2^w} \left(1-\frac{2^{w-1}(1/2-4\delta)n}{2Kn}\right)^{\tau_{2^{w+1}}-2^w} \\
&\le \left(\frac{2^{w+2}}{K}\right)^{2^w}\exp (-(\tau_{2^{w+1}}-2^w) 2^{w-4}/K).
\end{align*} 

Let $\ell = \max \tilde{I}$. By taking a union bound over the possible sets $\tilde{H}_\ell \subseteq [1,\tau_{2^{\ell+1}}]$ of size $2^\ell$ which include $\tau_{2^{\ell+1}}$, we see that
\begin{align*}
    &\mathbb{P}(T_J>\kappa \mid \mathcal{F}_{\ell-1}) \\
    &\le \sum_{\tau_{2^{\ell+1}} > \kappa} \mathbb{P}(T_{2^{\ell+1}} = \tau_{2^{\ell+1}} \mid \mathcal{F}_{\ell-1}) \\
    &\le \sum_{\tau_{2^{\ell+1}} > \kappa} \binom{\tau_{2^{\ell+1}}}{2^\ell-1}\left(\frac{2^{\ell+2}}{K}\right)^{2^\ell}\exp\left(-(\tau_{2^{\ell+1}}-2^\ell) 2^{\ell-4}/K\right)\\
    &\le \sum_{\tau_{2^{\ell+1}} > \kappa} 
    \frac{2^\ell}{\tau_{2^{\ell+1}}-2^\ell+1}\left(\frac{4e\tau_{2^{\ell+1}}}{K}\right)^{2^\ell} \exp\left(-(\tau_{2^{\ell+1}}-2^\ell) 2^{\ell-4}/K\right).
\end{align*}

Note now that $\ln(4ex)-x/32 \le -x/64$ for $x \ge 600$ and, hence, for $\tau_{2^{\ell+1}} \ge \max(2^{\ell+1}, \kappa) \ge 900K$, 
\[
    2^\ell \ln(4e\tau_{2^{\ell+1}}/K) - (\tau_{2^{\ell+1}} -2^\ell)2^\ell/(16K) 
    \le 2^\ell \ln(4e\tau_{2^{\ell+1}}/K) - \tau_{2^{\ell+1}} 2^\ell/(32K) 
    \le -2^{\ell}\tau_{2^{\ell+1}}/(64K).
\]

Hence, summing over all possible choices of $\ell$,
\begin{align*}
    &\mathbb{P}(T_J>\kappa) \\
    &\le \sum_{\ell \ge 1} \sum_{\tau_{2^{\ell+1}} > \kappa} 
    \frac{2^\ell}{\tau_{2^{\ell+1}}-2^\ell+1}\left(\frac{4e\tau_{2^{\ell+1}}}{K}\right)^{2^\ell} \exp\left(-(\tau_{2^{\ell+1}}-2^\ell) 2^{\ell-4}/K\right)\\
    &\le \sum_{\ell \ge 1} \sum_{\tau_{2^{\ell+1}} > \kappa}  \frac{2^\ell}{\tau_{2^{\ell+1}}-2^\ell+1}\exp\left(-2^\ell\tau_{2^{\ell+1}}/(64K)\right) \\
    &\le \sum_{\ell \ge 1: 2^\ell < K} \frac{2^{\ell+1}}{\kappa} \frac{\exp\left(-2^\ell\kappa / (64K)\right)}{(1-\exp(-2^\ell/(64K)))}\\ 
    &\qquad \qquad +  \sum_{\ell \ge 1: 2^\ell \ge K} \frac{2^{\ell}}{(\tau_{2^{\ell+1}}-2^\ell+1)} \frac{\exp\left(-2^\ell\kappa / (64K)\right)}{(1-\exp(-2^\ell/(64K)))}\\
    &\le \sum_{\ell \ge 1: 2^\ell < K} \frac{256K}{\kappa}  \exp\left(-2^\ell\kappa / (64K)\right) + 80 \sum_{\ell \ge 1: 2^\ell \ge K} \exp\left(-2^\ell\kappa / (64K)\right)\\\
    &\le 2\exp(-\kappa/(128K)). 
\end{align*}  
Here we used the simple estimates $1-\exp(-x) \ge x/2$ for $x\le 1/64$ and $\frac{2^{\ell}}{(\tau_{2^{\ell+1}}-2^\ell+1) (1-\exp(-2^\ell/(64K)))} \le 80$ for $\tau_{2^{\ell+1}} \ge 2^{\ell+1} \ge 2K$. 
\end{proof}

As a corollary of Lemma~\ref{lem:ball}, we now show that there is a set of at most $900(K+\ln n)$ colors such that there is a component of order at least $n/4$ formed by these colors in any proper edge-coloring of $K_n$ with at most $Kn$ colors. The key additional ingredient is contained in the following lemma.

\begin{lem}\label{lem:matching}
Assuming the setup and conclusion of Lemma~\ref{lem:ball}, 
there exists a set $\hat{S}'$ of $900\ln n$ colors such that there is a connected component using colors in $\hat{S}'\cup \hat{S}$ of order at least $n/4$.
\end{lem}

\begin{proof}
We say that a vertex $v$ is {\it good} if the connected component $C(v)$ containing this vertex formed by edges with colors in $\hat{S}$ has $|N^+(3/8;C(v))|\ge Kn/128$. 
We run a process where, at each time step $i \ge 1$, we choose a random color $s_i$ and consider the connected components formed by $\{s_{i'}:i'\le i\}\cup \hat{S}$. Let $N_i$ be the number of good connected components at step $i$, where a component is \emph{good} if it contains only good vertices. We stop the process once there is a component of order at least $n/4$, so we can assume that there is no component of order at least $n/4$ until $900\ln n$ colors have been sampled. For each good component $C(v)$ at step $i$, let $P$ denote the set of good vertices outside $C(v)$. We have $|P| \ge n - n/4 - n/10 > 5n/8$, as $|C(v)|\le n/4$ and the number of vertices which are not good is at most $n/10$. Thus, the number of colors between $C(v)$ and $P$ is at least $|N^+(3n/8;C(v))|\ge Kn/128$.  Hence, picking a uniformly random color, the expected number of good components with an edge in that color to a different good component is at least $(Kn/128)/(2Kn) \cdot N_i = N_i/256$. Therefore, there exists a color for which at least $N_i/256$ good components have an edge in that color to a different good component. This implies that a color can be chosen to guarantee that $N_{i+1} \le N_i(1-1/512)$. Hence, we can choose the colors so that $1 < N_{i}\le n(1-1/512)^i$ for all $i\ge 1$, which is a contradiction if $i\ge 900\ln n$.
\end{proof}

Combining Lemmas \ref{lem:ball} and \ref{lem:matching}, we obtain the promised result.

\begin{lem}\label{lem:connect}
Consider a proper edge-coloring of the complete graph on a set $A$ of $n$ vertices using at most $2Kn$ colors, where each color class contains at most $n/K$ edges. Then there exist $900 (K+\ln n)$ colors such that there exists a component $A_c$ of order at least $n/4$ in these colors in $A$. 
\end{lem}

\subsection{Improved counting} \label{sec:impcount}

We now prove Theorem \ref{thm:improvedcounting} in the following slightly reworded form. 

  \begin{thm}\label{uptoaconstantoptimal}
There is an absolute constant $C$ such that, for any proper edge-coloring of $K_N$, the number of vertex subsets $A$ of order $n$ spanning at most $Kn$ colors is at most $N^{C( K+\ln n)}(C K)^n$. 
\end{thm}

A key part in the proof of this theorem is played by structures that we call efficiently colored trees. Given a vertex subset $A$ of order $n$ in a proper edge-colored complete graph, an \emph{efficiently colored tree} on a subset $A_c$ of $A$ is a tree together with an edge-coloring using an ordered list of $900(K+\ln n)$ colors with the following properties:
\begin{itemize}
\item 
There is a partition of $A_c$ into components $B_1,\dots,B_\ell$, each connected by edges with the first $900K$ colors. 
\item 
After revealing $i$ of the remaining $900\ln n$ colors and the corresponding edges between the connected components, 
the number of connected components is at most $(1-1/512)^i n$. 
\end{itemize}
The importance of efficiently colored trees for us stems from the following result, which has essentially the same proof as Lemma~\ref{lem:matching} above.

\begin{lem}\label{lem:efficient-tree}
In any proper edge-coloring of a complete graph, for every vertex subset $A$ of order $n$ which spans at most $Kn$ colors, 
there exists a subset $A_c$ of $A$ of order at least $n/4$ and an efficiently colored tree on $A_c$. 
\end{lem}

We now give the proof of Theorem \ref{uptoaconstantoptimal}. 

\begin{proof}[Proof of Theorem \ref{uptoaconstantoptimal}]
By Lemma \ref{lem:efficient-tree}, for any set $A$ of the required form, there is a subset $A_c$ of $A$ which admits an efficiently colored tree. We will bound the number of sets $A$ by first bounding the number of efficiently colored trees, thereby giving an upper bound on the number of possible choices for $A_c$. We then bound, for each $A_c$ admitting an efficiently colored tree, the number of sets $A$ extending $A_c$. 

We first count the number of efficiently colored trees. Each efficiently colored tree can be specified as follows:
\begin{enumerate}
    \item Select a set $S_1$ of $900K$ colors and a set $S_2$ of $900\ln n$ colors.
    \item Select an unlabelled tree on $n'\in [n/4,n]$ vertices (in at most $3^n$ ways for $n$ sufficiently large by a theorem of Otter \cite{Otter}) and a partition of the tree into components $B_1,\dots,B_\ell$. 
    \item Select a color for each edge in the components $B_i$ using the colors in $S_1$. 
    \item For each $0\le i \le 900\ln n$, we keep track of a set of $\ell_i$ components, starting with the components $B_1,\dots,B_\ell$ for $i=0$. There are $\ell_i-1$ edges of the tree between the components. We then select a subset of these $\ell_i-1$ edges and assign the edges in this set the $i$-th color in $S_2$. 
\end{enumerate}
Upon performing the above steps, we obtain a tree on $n'$ vertices together with a coloring of its edges. We only consider such trees where $\ell_i$ satisfies $\ell_i\le (1-1/512)^{i}n$. As the edge-coloring is proper, for each such colored tree, upon specifying the choice of the root of the tree among the $N$ vertices of $K_N$, there is at most one way to identify all remaining vertices of the tree with vertices in $K_N$. Thus, the number of efficiently colored trees, and hence the number of choices for $A_c$, is at most 
\[
\binom{N^2}{900K}\binom{N^2}{900\ln n} \cdot 3^n \cdot \binom{n'-1}{\ell-1} \cdot (900K)^{n'-\ell} \cdot \prod_{i=0}^{900\ln n} 2^{\ell_i-1} \cdot N.
\]
Here, the first term $\binom{N^2}{900K}\binom{N^2}{900\ln n}$ corresponds to the number of choices for $S_1$ and $S_2$, the term $\binom{n'-1}{\ell-1}$ corresponds to the number of ways to select $\ell-1$ edges of the tree partitioning it into $\ell$ components, the term $(900K)^{n'-\ell}$ is an upper bound on the number of ways to color each edge in the components $B_i$ with a color in $S_1$, the terms $2^{\ell_i-1}$ correspond to the number of ways to choose the subsets of the sets of $\ell_i-1$ edges to be colored with the $i$-th color in $S_2$ and the last term $N$ is the number of ways to assign the root of the tree to a vertex in $K_N$. 

Next, we identify $A$ from $A_c$. Starting from $A'=A_c$, we consider the following process. Let $N(A')$ be the set of colors between vertices in $A'$. If there exists a vertex $a$ in $A$ for which there are at least $n/8$ colors between $a$ and $A'$ which are not contained in $N(A')$, then we include $a$ in $A'$, noting that there are at most $N$ choices for $a$. This can occur at most $8 K$ times. 
In the end, we obtain a set $A'$ such that, for all $a\in A$, there are at least $n/8$ edges between $a$ and $A'$ with color in $N(A')$. Let $Y$ be the set of all vertices in $K_N$ such that the number of edges between this vertex and $A'$ with color in $N(A')$ is at least $n/8$. We have $|Y|\le \frac{n \cdot |N(A')|}{n/8} = 8|N(A')|\le 8Kn$. The number of choices for the remaining vertices of $A$, given that they are all in $Y$, is therefore at most $\binom{|Y|}{n-n'} \le (8eKn/(n-n'))^{n-n'}$. 

Hence, the number of sets $A$ is at most 
\begin{align*}
&\sum_{n'\in [n/4,n], \ell\le n'} N^{1800(K+\ln n)} 3^n \binom{n'-1}{\ell-1} (900K)^{n'-\ell} \cdot 2^{900n} \cdot N \cdot N^{8 K} \cdot  (8eKn/(n-n'))^{n-n'} \\
&\le \sum_{n'\in [n/4,n]} N^{1800(K+\ln n)} 3^n (900K)^{n'-1} \cdot 2^{950n} \cdot N^{8 K+1} \cdot  (8eKn/(n-n'))^{n-n'}\\
&\le N^{1800(K+\ln n)} (2^{1000}K)^n N^{20K},
\end{align*}
as required.
\end{proof}

We note the following immediate corollary counting the number of subsets spanning few colors in locally-bounded colorings. The proof, which we omit, is almost exactly the same as that of Theorem~\ref{complete}, but with Theorem~\ref{uptoaconstantoptimal} replacing Theorem~\ref{proper} as the main input. 

\begin{cor}\label{Deltaboundedcor}
There is an absolute constant $C$ such that in any edge-colored complete graph on $N$ vertices in which each color class has maximum degree at most $\Delta$, the number of subsets of $n$ vertices in which at most $Kn$ colors appear is at most $N^{C(\Delta K+\ln n)} (C\Delta K)^n$.
\end{cor}

The following improvement of Corollary~\ref{countgroup}, which includes Theorem~\ref{thm:count-group}, is now an immediate corollary.

\begin{cor}\label{countgroup3}
There is a constant $C$ such that if $G$ is a finite group of order $N$, then the number of subsets $A \subset G$ with $|A|=n$ and $|AA| \le K|A|$ is at most $N^{C(K + \ln n)}(CK)^{n}$. Similarly, the number of subsets $A \subset G$ with $|A|=n$ and $|AA^{-1}| \le K|A|$ is at most $N^{C(K + \ln n)}(CK)^{n}$.
\end{cor}

Though we will not go into the matter in detail, it is again possible, as in Theorem~\ref{expand}, to extend our results to $\lambda$-edge-expanders, a result which in turn has the following corollary.

\begin{cor}\label{contgroup4}
There is a constant $C$ such that if $G$ is a finite group of order $N$, then the number of pairs of subsets $A,B \subset G$ with $|A|=|B|=n$ and $|AB| \le K|A|$ is at most $N^{C(K + \ln n)}(CK)^{2n}$.
\end{cor}

\subsection{The number of colors required to connect the entire graph} \label{subsectionconnectedspanning}

In this subsection, we take a look at the number of colors needed to find a connected spanning subgraph in a proper edge-coloring of $K_n$ with at most $Kn$ colors. We begin with a simple observation that, together with the results of the previous subsections, will give our upper bound.

\begin{lem}\label{lem:grow-comp}
    Given a proper edge-coloring of $K_n$ with at most $Kn$ colors, a subset $S$ of $V(K_n)$ of order $s$ and any $\delta > 0$, there exist $(n/s)K\ln(1/\delta)$ colors such that all but at most $\delta n$ vertices of $K_n$ are connected to $S$ by a path using these colors. 
\end{lem}

\begin{proof}
    Starting with $S_0 = S$, for every $i\ge 0$, the number of edges between $S_i$ and the remaining vertices is $|S_i|(n-|S_i|)$ and thus there is a color which connects at least $|S_i|(n-|S_i|)/(Kn)$ new vertices to $S_i$. Define $S_{i+1}$ to be the union of $S_i$ and these new vertices. We can then guarantee that $n-|S_{i+1}| \le (n-|S_i|)(1-|S_i|/(Kn)) \le (n-|S_i|)(1-(s/n)/K)$. Thus, after at most $t = (n/s)K\ln(1/\delta)$ steps, we obtain that the complement of $S_{t}$ has size at most $\delta n$. 
\end{proof}

\begin{thm}\label{thm:spanning-tree}
In any proper edge-coloring of $K_n$ with at most $Kn$ colors, there is a spanning tree containing at most $1000 K\ln(n/K)$ colors. 
\end{thm}

\begin{proof}
By Lemma \ref{lem:connect}, we can find a set of $900(K+\log n)$ colors that connects a component of order at least $n/4$.  By Lemma \ref{lem:grow-comp} with $\delta = K/n$, by using an additional $4K\ln(n/K)$ colors, we obtain a component whose complement has order at most $K$. By using $K$ additional colors, we can thus connect all $n$ vertices. This implies that there is a spanning tree using at most $900(K+\log n) + 4K\ln(n/K) + K < 1000K\ln(n/K)$ 
colors, yielding the required result. 
\end{proof}

We now give a construction which shows that the bound in Theorem~\ref{thm:spanning-tree} is tight up to a constant factor.

\begin{lem}\label{lem:manycolors}
For any sufficiently large $n$ and $1 \leq K \leq n/(\ln n)^8$, there is an edge-coloring of the complete graph $K_n$ with at most $Kn$ colors such that every connected spanning subgraph has at least $\frac{1}{50}K\ln(n/K)$ colors. 
\end{lem}

\begin{proof}
Let $s=\sqrt{Kn}$. We show that there exists a proper edge-coloring of the complete graph $K_{n+s}$ with at most $5Kn$ colors for which any spanning tree uses at least $M:=\frac{1}{8}K\ln(n/K)$ colors. The conclusion of the lemma follows from rescaling the parameters by appropriate factors.  

For disjoint vertex subsets $A$ and $B$, we say that a set $S$ of colors connects $A$ to all of $B$ if each vertex in $B$ has at least one incident edge whose other vertex is in $A$ and whose color is in $S$. Partition the vertex set of $K_{n+s}$ into parts $A_n$ of size $n$ and $A_s$ of size $s$. We will construct a proper edge-coloring of the complete bipartite graph $K_{s,n}$ with parts $A_s$ and $A_n$ using at most $Kn+2s^2/K =(K+2)n$ colors such that no set of $M$ colors connects $A_n$ to all of $A_s$. We then properly edge-color the complete graph on $A_n$ arbitrarily with a new set of $n$ colors and all edges in the complete graph on $A_s$ with $\binom{s}{2} \leq Kn/2$ new colors. In total, we will use at most $(K+2)n+n+Kn/2 \leq 5Kn$ colors. 

If there is a spanning subgraph that uses at most $M$ colors, then we may assume it is a spanning tree. Furthermore, we can replace edges of the tree that have both vertices in $A_s$ by edges in the complete bipartite graph between $A_s$ and $A_n$ to get a spanning tree that still uses at most $M$ colors and does not have any edge with both vertices in $A_s$. It follows that there would be a set of $M$ colors that connects $A_n$ to all of $A_s$, contradicting the claimed property. It therefore suffices to show that there is an edge-coloring of the complete bipartite graph $K_{s,n}$ with parts $A_s$ and $A_n$ using at most $Kn+2s^2/K$ colors such that no set of $M$ colors connects $A_n$ to all of $A_s$. 

First consider the following random edge-coloring of $K_{s,n}$. For each vertex $v$ in $A_s$, pick a uniformly random proper coloring of the edges adjacent to $v$ from a fixed set of $Kn$ colors. We refer to this random edge-coloring as the original edge-coloring of $K_{s,n}$. For each edge $e$, declare that $e$ is bad if there exists an edge $e'$ adjacent to $e$ (necessarily at the vertex of $e$ that lies in $A_n$) such that $e'$ is assigned the same color as $e$. For each bad edge, we assign to that edge a distinct color from a new set of colors. After recoloring the bad edges in this way, we obtain a random proper edge-coloring of $K_{s,n}$, which we refer to as the new edge-coloring. 

For each edge $e$, the colors of the edges adjacent to $e$ on the $A_n$ side are independent choices each taken from a set of $Kn$ colors. Thus, the probability that $e$ is not bad is  
\[
(1-1/(Kn))^{s-1} \ge 1-s/(Kn).
\]
Therefore, the expected number of bad edges is at most $sn \cdot \frac{s}{Kn} = s^2/K$. Hence, the expected number of additional colors used to color $K_{s,n}$ is at most $s^2/K$. Thus, by Markov's inequality, with probability at least $1/2$, the number of additional colors used to color $K_{s,n}$ is at most $2s^2/K$. 

Next, we bound the probability that there exist $M$ colors in the original coloring which connect $A_n$ to all of $A_s$. Note that if the new edge-coloring has $M$ colors which connect $A_n$ to all of $A_s$, then so did the original edge-coloring using the same set of edges.  
Fix a set $S$ of $M$ colors. In the original coloring, the probability that a given vertex in $A_s$ is not incident to an edge whose color is in $S$ is at least $\left(1-\frac{M}{Kn}\right)^{n} \geq e^{-2M/K}$, 
where, for the inequality, we used 
that $0 < \frac{M}{Kn} < 1/8$. 
Therefore, the probability that a given vertex in $A_s$ is incident to at least one edge of color in $S$ is at most $1-e^{-2M/K}$. These events are independent over different vertices in $A_s$, so the probability that $S$ connects $A_n$ to all of $A_s$ is at most $$(1-e^{-2M/K})^s \leq e^{-se^{-2M/K}}=e^{-K(n/K)^{1/4}}.$$ 

Taking a union bound over all ${Kn \choose M}$ possible sets $S$ of $M$ colors, the probability that there exists a set of $M$ colors that connects $A_n$ to all of $A_s$ is at most $${Kn \choose M}e^{-K(n/K)^{1/4}}\leq \left(\frac{eKn}{M}\right)^M e^{-K(n/K)^{1/4}} \leq n^M
e^{-K(n/K)^{1/4}}\leq 
1/e.$$
The second to last inequality follows from the choice of $M$ and the conditions in the lemma statement on $n$ and $K$. To see the last inequality, note that by taking the natural logarithm it suffices to show that $M\ln n - K(n/K)^{1/4} \leq -1$. As $M=\frac{1}{8}K\ln(n/K)$, by dividing by $K$ and reorganizing, it suffices to show that $(n/K)^{1/4} \geq \frac{1}{8}\ln(n/K)\ln n+1/K$, which would follow from $(n/K)^{1/4} \geq \ln(n/K)\ln n$. By our conditions on $n$ and $K$, $n/K$ can be taken sufficiently large that $(n/K)^{1/8} \geq \ln(n/K)$ and, by the upper bound on $K$, we have $(n/K)^{1/8} \geq \ln n$. Hence, the desired inequality holds.

Therefore, with probability at least $1-1/2-1/e>0$, the new edge-coloring of the complete bipartite graph $K_{s,n}$ has the desired properties, completing the proof. 
\end{proof}

\subsection{Improved bounds on additive dimension}
\label{adddimsubsect}

In this subsection, we will make use of some of the ideas from the previous subsections to prove Theorem~\ref{thm:dim-strong}. Recall the statement, that if $A$ is a subset of an abelian group $G$ of size $n$ with $|A-A| \le K|A|$, then $d_*(A) = O(K\log \log n+ \log n)$. Moreover, if $G$ has exponent $r$, then $d_*(A) = O(K \log r + \log n)$. Most of the proof works over a general group, so, to stay in keeping with the results and proofs on which it builds, we will stay in this setting until the very end, only specializing to the case where the group is abelian when we need to. 

\begin{proof}[Proof of Theorem~\ref{thm:dim-strong}.]  Consider a group $G$ and a subset $A\subseteq G$ with $|A|=n$ and $|AA^{-1}|\le Kn$. As in the proof of Corollary~\ref{cor:dim-weak}, we consider the edge-coloring of the complete graph on vertex set $A$ which assigns to each edge $\{x,y\}$ the color $\{xy^{-1},yx^{-1}\}$. This is an edge-coloring where each color class has maximum degree at most two, so  Vizing's theorem implies that by increasing the number of colors by a factor of at most three, we may obtain a proper edge-coloring of the complete graph on $A$.

    Applying Lemma \ref{lem:ball-key} with $J = \delta K$ for $\delta = 1/16$ and $\kappa = 900K$, we obtain, as in the proof of Lemma~\ref{lem:ball}, a set $S_1$ of $\kappa$ colors such that for a set $\mathcal{G}_0$ of vertices $v$ with $|\mathcal{G}_0| \ge (1-2\exp(-\kappa/(128K)))n$, we can find a set $B(v)$ of size $J$ such that each vertex of $B(v)$ can be connected to $v$ using a sequence of edges of color in $S_1$ with length at most $\lceil \log J \rceil$. Indeed, consider the procedure in Lemma~\ref{lem:ball} for a vertex $v$ such that the event in Lemma \ref{lem:ball-key} holds. Letting $B_j = B_j(v)$, every vertex $u\in B_j\setminus B_{j-1}$ is connected to a vertex in $B_{2^{\lfloor \log_2(j-1)\rfloor}}$ through an edge with color in $S_1$. Hence, inductively, we have that every vertex in $B_{2^w}$ is connected to $v$ through a sequence of edges with color in $S_1$ of length at most $w$. Thus, every vertex in $B(v) = B_J(v)$ can be connected to $v$ using a sequence of edges of color in $S_1$ with length at most $\lceil \log J \rceil$. 
    Furthermore, $|N^+(3/8; B(v))| \ge Kn/128$.   

    Recall that each color $s$ can be identified with a set of group elements $\{g,g^{-1}\}$. We may also identify each color $s$ with an arbitrarily chosen element from the corresponding set of group elements. Abusing notation, we will sometimes replace a set of colors, such as $S_1$, by the subset of $G$ that its elements identify with. Using this identification, set
    \[
        \mathcal{B}(t) = \left\{\prod_{i=1}^{t\log J} q_i: q_i\in S_1\cup S_1^{-1} \cup \{1\}\right\}.
    \] 

   We will need the following key claim.
    
    \begin{claim}\label{claim:log-n}
        There exists a sequence $S_2$ of at most $L = 1536\ln n$ elements of $G$, a set $\mathcal{G} \subseteq A$ of size at least $n/4$ and a vertex $v_1\in \mathcal{G}$ such that 
        \[
            \mathcal{G} \subseteq \left(\prod_{s\in S_2} \{s,1\}\right) v_1 \mathcal{B}(2L).
        \]
    \end{claim}
    
    \begin{proof}
    Consider an arbitrary vertex $v_1 \in \mathcal{G}_0$. Let $\mathcal{C}_0 = \{v_1\}$. For $i\in [1, L-1]$, we consider the following process, which defines a set $\mathcal{C}_i\subseteq \mathcal{G}_0$ for which the sets $x \mathcal{B}(L-i)$ are disjoint for all $x\in \mathcal{C}_i$. Let $\overline{\mathcal{C}_{i-1}} = \mathcal{C}_{i-1} \mathcal{B}(L-i)^2 \cap \mathcal{G}_0$. For a suitable $s_i \in G$ to be chosen later, we will define $\mathcal{C}_{i}$ to be a set containing all $x\in \mathcal{C}_{i-1}$ and, for each $x\in \mathcal{C}_{i-1}$ such that $s_i B(x) \cap (\mathcal{G}_0 \setminus \overline{\mathcal{C}_{i-1}}) \ne \emptyset$, exactly one vertex $v \in s_i B(x) \cap (\mathcal{G}_0 \setminus \overline{\mathcal{C}_{i-1}})$.
    
    We first observe that the sets $x\mathcal{B}(L-i)$ are disjoint for all $x\in \mathcal{C}_i$. Indeed, suppose that there exist $x,y\in \mathcal{C}_i$ such that $x\mathcal{B}(L-i) \cap y\mathcal{B}(L-i) \ne \emptyset$. Any $x \in \mathcal{C}_i$ can be written in the form $s_i^{c_x}x^* b_x$ for $c_x \in \{0,1\}$, $x^* \in \mathcal{C}_{i-1}$ and $b_x\in (x^*)^{-1} B(x^*) \subseteq \mathcal{B}(1)$. Furthermore, note that we can take $b_x = 1$ if $c_x = 0$. Thus, we can find $x^*$, $y^* \in \mathcal{C}_{i-1}$ and $b_x \in (x^*)^{-1} B(x^*)$, $b_y \in (y^*)^{-1} B(y^*)$ such that $s_i^{c_x}x^* b_x \mathcal{B}(L-i) \cap s_i^{c_y} y^* b_y \mathcal{B}(L-i) \ne \emptyset$. If $c_x=c_y$, then we have that $x^* \mathcal{B}(L-i+1)\cap y^* \mathcal{B}(L-i+1) \ne \emptyset$, which cannot occur since the sets $x \mathcal{B}(L-i+1)$ are disjoint for $x\in \mathcal{C}_{i-1}$. Otherwise, suppose, without loss of generality, that $c_x=1$ and $c_y=0$. Then $x = s_i x^*b_x$ is contained in $y \mathcal{B}(L-i)\mathcal{B}(L-i)^{-1} = y \mathcal{B}(L-i)^2 \subseteq \overline{\mathcal{C}_{i-1}}$ for $y\in \mathcal{C}_{i-1}$, contradicting $x \in \mathcal{G}_0 \setminus \overline{\mathcal{C}_{i-1}}$. This establishes the desired conclusion. 
    
    If $|\overline{\mathcal{C}_{i-1}}| < n/4$, we claim that we can find $s_i$ such that $|\mathcal{C}_{i}|-|\mathcal{C}_{i-1}| \ge |\mathcal{C}_{i-1}|/768$. Indeed, since $|\mathcal{G}_0 \setminus \overline{\mathcal{C}_{i-1}}| \ge 5n/8$, for each $x\in \mathcal{C}_{i-1}$, we have $|N(B(x), \mathcal{G}_0 \setminus \overline{\mathcal{C}_{i-1}})| \ge Kn/128$. Thus, picking a uniformly random color $c$, the expected number of $x\in \mathcal{C}_{i-1}$ for which there is an edge in color $c$ between $B(x)$ and $\mathcal{G}_0 \setminus \overline{\mathcal{C}_{i-1}}$ is at least $|\mathcal{C}_{i-1}| / 384$ (recall that the proper edge-coloring we constructed uses at most $3Kn$ colors). Observing that each edge in color $c$ corresponds to either a group element $s_i$ or its inverse, we obtain that there exists $s_i$ such that the number of $x\in \mathcal{C}_{i-1}$ for which $s_i B(x) \cap (\mathcal{G}_0 \setminus \overline{\mathcal{C}_{i-1}}) \ne \emptyset$ is at least $|\mathcal{C}_{i-1}| / 768$. Finally, note that, for $x,y\in \mathcal{C}_{i-1}$, we have $s_i B(x) \cap s_i B(y) = \emptyset$, since $B(x)\subseteq x\mathcal{B}(1) \subseteq x\mathcal{B}(L-i)$ (recalling that we only consider $i\le L-1$). Hence, we can find $s_i$ such that $|\mathcal{C}_{i}|-|\mathcal{C}_{i-1}| \ge |\mathcal{C}_{i-1}|/768$. 
    
    By our choice of $L = 1536\ln n$, we cannot have $n\ge |\mathcal{C}_{i}| \ge (1+1/768)|\mathcal{C}_{i-1}|$ for all $i\le L-1$. Thus, there is some $i\le L-1$ such that $|\overline{\mathcal{C}_{i-1}}| \ge n/4$. We then define $\mathcal{G} = \overline{\mathcal{C}_{i-1}}$, so that 
    \[
        \mathcal{G} \subseteq \left(\prod_{j=1}^{i} \{s_j, 1\} \right) v_1 \mathcal{B}(2L). 
    \]
    Taking $S_2 = \{s_1, s_2, \dots\}$ gives the required conclusion.
    \end{proof}

    Let $T = |S_2| \le 1536 \ln n$. Set 
    \[
        \mathcal{R} = \left(\prod_{s\in S_2} \{s,1\}\right) v_1 \mathcal{B}(2L)
    \]
    and, for $t\ge 1$, 
    \[
        \mathcal{R}(t) = (\mathcal{R}^{-1} \cup \mathcal{R})^t.
    \]
    By Claim \ref{claim:log-n}, we have that 
    $\mathcal{G} \subseteq \mathcal{R}(1)$  
    and $\mathcal{G}^{-1}\subseteq \mathcal{R}(1)$.  
    
    For each vertex $u$ in $A\setminus \mathcal{G}$, we say that $u$ is \emph{nearly covered} if $u$ has a neighbor in $\mathcal{G}$ through an edge with color in $\mathcal{G}^{-1}\mathcal{G}$. Thus, for each vertex $u$ which is not nearly covered, the colors between $u$ and $\mathcal{G}$ are distinct from the colors in $\mathcal{G}$. Observe also that the set of nearly covered vertices is contained in $\mathcal{R}(3)$. 
    Consider a vertex $u_1$ which is not nearly covered. For $X,Y\subseteq A$, let $N(X,Y)$ denote the set of colors on the edges between $X$ and $Y$. Let $c_1$ be the color of an edge between $u_1$ and $\mathcal{G}$. Define $\mathcal{G}_1$ to be the set of not nearly covered vertices $u'$ which have $N(u', \mathcal{G}) \cap N(u_1,\mathcal{G})\ne \emptyset$. Let $v', v\in \mathcal{G}$ be such that $u'(v')^{-1} = u_1v^{-1}$ or $u'(v')^{-1} = vu_1^{-1}$. If $c_1 = u_1^{-1} \bar{v}$ for $\bar{v}\in \mathcal{G}$, we can then write $u' = u_1 v^{-1}v' = \bar{v}c_1^{-1} v^{-1} v'$ or $u' = vu_1^{-1}v'=vc_1\bar{v}^{-1}v'$. As such, $\mathcal{G}_1$ must be contained in 
    \[
        \mathcal{G} \cdot \{c_1, c_1^{-1}\}\cdot  \mathcal{G}^{-1}\mathcal{G} \subseteq \mathcal{R}(1) \cdot \{c_1, c_1^{-1}\} \cdot \mathcal{R}(2).
    \]
    
    Given vertices $u_1,\dots,u_h$, we define $\mathcal{G}_h$ to be the set of vertices $u'$ that are not nearly covered, $N(u',\mathcal{G}) \cap N(u_h, \mathcal{G}) \ne \emptyset$ and $N(u',\mathcal{G}) \cap (\bigcup_{i<h}N(u_i, \mathcal{G})) = \emptyset$. Let $c_h$ denote an arbitrary color between $u_h$ and $\mathcal{G}$. As above, we have that 
    \[
        \mathcal{G}_h \subseteq (\bigcup_{i<h} \mathcal{G}_{i}) \cup \left(\mathcal{R}(1) \cdot \{c_h,c_h^{- 1}\}\cdot  \mathcal{R}(2)\right). 
    \]
    If there exists a vertex which is not nearly covered and which is not contained in $\mathcal{G}_i$ for any $i\le h$, we then take $u_{h+1}$ to be an arbitrary such vertex and continue the process.  Observe that the $N(u_i, \mathcal{G})$ are pairwise disjoint and each has size at least $|\mathcal{G}|\ge n/4$. Thus, the number of steps is bounded above by $O(K)$. 
    We therefore conclude that there is a set of colors $S_3$ of size $O(K)$ such that 
    \[
        A \subseteq \bigcup_{s_3\in S_3} \left(\mathcal{R}(1)\cdot \{s_3,s_3^{- 1}\}\cdot \mathcal{R}(2)\right). 
    \]

    Upon specializing to the case where $G$ is an abelian group, we observe that each element of $A$ can be represented as a sum of elements of $\pm (S_1\cup S_2\cup S_3)$ where the multiplicity of each element in $\pm S_3$ is at most $1$, the multiplicity of each element in $\pm S_2$ is $O(1)$ 
    and the multiplicity of each element in $\pm S_1$ is at most $O((\log n)(\log K))$. By selecting $\tilde{S}_1$ to include all multiples $2^bs_1$ for $s_1\in S_1$ and $2^b \le O((\log n)(\log K))$, we can then represent every element of $A$ as a $\{-1,0,1\}$-combination of elements in the multiset $\pm(\tilde{S}_1\cup S_2\cup S_3)$ (where the elements in $S_2$ have $O(1)$ multiplicity), which has size $O(K\log ((\log n)(\log K)) + \log n) = O(K\log \log n + \log n)$. Hence, $d_*(A) = O(K\log \log n+\log n)$. Furthermore, if $G$ has bounded exponent $r$, then we only need to include in $\tilde{S}_1$ all multiples $2^bs_1$ for $s_1\in S_1$ and $2^b \le r$, yielding $d_*(A) = O(K\log r + \log n)$. 
\end{proof}

\section{Alon's conjecture}\label{sec:Alon}

In this section, we prove several results relating to Alon's conjecture that there exists an absolute constant $C$ such that every finite group has a $C$-Ramsey Cayley graph. 
In Section~\ref{sec:Orlitsky}, we warm up by showing that there are self-complementary Ramsey Cayley graphs, which already answers a question of Alon and Orlitsky~\cite{AO95}. Then, in Sections~\ref{sec:coprime6} and~\ref{sec:othergroups}, we address Alon's conjecture directly by first proving that it holds for all abelian groups of order coprime to $6$ and then lifting this result to all groups with a large abelian subgroup of order coprime to $6$. In particular, these results imply Corollary~\ref{cor3}, that Alon's conjecture holds for abelian groups of almost all orders. 
Section~\ref{sec:RamseyCayley} contains the proofs of Theorems~\ref{thm:Fp} and~\ref{thm:Fq}. Recall that the first of these states that if $p \ge 5$ is a prime and $V$ is a vector space V of order $N$ with characteristic $p$, then there is a Cayley graph on $V$ in which the clique and independence numbers are both at most $(2 + o(1)) \log_2 N$, while the second says that if the characteristic is $1 \pmod 4$, then we may even take the Cayley graph to be self-complementary. Both proofs work in far more general contexts, which we explain in detail in the section itself. Finally, in Section~\ref{subsec:EZL}, we prove a technical result needed for the proofs of Theorem~\ref{thm:Fp} and~\ref{thm:Fq} giving an essentially tight upper bound on the number of subsets $A$ of $\mathbb{F}_p^n$ of size $t$ with $|A-A| \leq Kt$. The proof is a modification of the proof of a similar result on $A+A$ proved by Even-Zohar and Lovett~\cite{EZL}, but may nevertheless be of independent interest.

\subsection{Self-complementary Ramsey Cayley graphs} \label{sec:Orlitsky}

In this short subsection, we answer a question of Alon and Orlitsky~\cite{AO95} by showing that there are self-complementary Ramsey Cayley graphs. The proof of this result already contains many of the ingredients we will use throughout the rest of the section, but in simplified form. 
Though we present the proof over groups of the form $\mathbb{Z}_5^d$, one can easily modify the construction to work over $\mathbb{Z}_p^d$ for any fixed prime $p \equiv 1\pmod 4$. 
A key point to note is that in order to obtain a Ramsey Cayley graph over $\mathbb{Z}_5^d$ we cannot simply take a uniform random Cayley graph, since such graphs have clique number $\Theta(\log N \log \log N)$ with high probability, where $N=5^d$ is the number of vertices.

\begin{thm}
For every $d$, there is a self-complementary Ramsey graph over $\mathbb{Z}_5^d$. \end{thm}

\begin{proof} Let $G = \mathbb{Z}_5^d$ and $N=|G|=5^d$. The nonzero elements of the cyclic subgroups of $G$ partition $G \setminus \{0\}$. The set of nonzero elements of a cyclic subgroup of $G$ is a set of the form $\{x,2x,3x,4x\}$. For each such cyclic subgroup, we pick exactly one of $\{x,4x\}$ or $\{2x,3x\}$ to be a subset of the generating set $S$ uniformly at random. This guarantees that the generating set $S$ has several properties: 
\begin{itemize}
    \item $S$ is symmetric.
    \item The Cayley graph $G_S$ is self-complementary through the isomorphism $\phi:\mathbb{Z}_5^d \to \mathbb{Z}_5^d$ given by $\phi(x)=2x$.
    \item For each nonzero $x \in \mathbb{Z}_5^d$, exactly one of $x$ and $2x$ is in $S$. 
\end{itemize} 
It therefore suffices to show that with high probability  the clique number of $G_S$ is $O(\log N)=O(d)$. 

We claim that every set $A \subset G$ which is a clique in $G_S$ satisfies $|A - A| \geq |A|^{4/3}$. For a clique $A$ in $G_S$, we have $A - A \subset S$ and, since $S$ does not contain both $x$ and $2x$ for any $x$, $A$ has no nontrivial solution to the equation $a_1-a_2=2(a_3-a_4)$ or, equivalently, to $a_1+2a_4=a_2+2a_3$. This implies that all the sums $a_1+2a_4$ with $a_1,a_4 \in A$ are distinct, so $A+2 \cdot A=|A|^2$. On the other hand, by the Plünnecke--Ruzsa inequality, $|A+2 \cdot A| \leq |3A| \leq (|A - A|/|A|)^3 |A|$. Hence, $|A - A| \geq |A|^{4/3}$, as claimed. 

To complete the proof, it suffices to show that, for some appropriate constant $C'$, the expected number of cliques of order $n=C'\log N$ in $G_S$ is $o(1)$. We will show that we may take $C'=100C$, where $C \geq 1$ is the constant in Theorem~\ref{thm:count-group}. By Theorem~\ref{thm:count-group}, the number of sets $A \subset G$ with $|A|=n$ and $|A-A|=m$ is at most $N^{C(K+\ln n)}(CK)^n$, where $K:=m/n$. For any possible clique $A$ of order $n$ in $G_S$, we have that $|A - A| \geq |A|^{4/3}$ and the probability that $A$ is a clique in $G_S$ is $2^{-(|A-A|-1)/2}$, where we used that nonzero distinct elements of $A-A$, no pair of which are inverses, appear in $S$ independently of each other. 
Therefore, the expected number of cliques in $G_S$ of order $n=C'\log N$ is at most $$\sum_{n^{4/3} \leq m \leq n^2}N^{C(m/n+\ln n)}(Cm/n)^n2^{(1-m)/2}=o(1).$$ 
Hence, the clique number of $G_S$ is less than $n$ with high probability. 
\end{proof}

The same proof works over $\mathbb{Z}_p^d$ for any fixed prime $p \equiv 1\pmod 4$, since in this case there exists $\alpha\in \mathbb{Z}_p$ with $\alpha^2 = -1$ and we can partition $\mathbb{Z}_p^d$ into disjoint sets of the form $\{x,\alpha x,\alpha^2x,\alpha^3x\}$. However, we will need some additional ideas when we prove the stronger and more general Theorem~\ref{thm:Fq} later in the section.

\subsection{Abelian groups with order coprime to $6$} \label{sec:coprime6}

In this subsection, we prove Alon's conjecture for abelian groups of order coprime to $6$. 

\begin{thm}\label{thm:coprime-6}
There is an absolute constant $C'$ such that, for every abelian group $G$ of order coprime to $6$, there is a Cayley graph on $G$ with no clique or independent set of order $C'\log |G|$.  
\end{thm}

For the rest of this subsection, let $G$ be an abelian group of order $N$ coprime to $6$. 
We next describe how we pick a random symmetric subset $S$ of $G$ and we will then show that the random Cayley graph $G_S$ with high probability  has no clique or independent set of order $C'\log |G|$. 

We first consider the equivalence relation $x \equiv -x$, which partitions the set of nonzero elements of $G$ into pairs. We then partition the family of such equivalence classes into disjoint orbits of the map $x\mapsto 2x$. We select a random subset $S$ of $G$ as follows. Partition the set of equivalences classes given by the orbit of $x$ into consecutive sets of size two 
$\{x,2x\},\{4x,8x\},\ldots$ and possibly one singleton (which happens if and only if the orbit is of odd size). For each of these sets of size two, we include exactly one of the two corresponding equivalence classes uniformly at random as a subset of $S$. If there is a singleton, we pick its equivalence class to be a subset of $S$ uniformly at random. Note that we make all of these random choices independently of each other. 

From this construction, it is clear that the distribution of the random Cayley graph $G_S$ and its complement are identical. In particular, each nonzero element of $G$ has probability exactly $1/2$ of being in $S$. Furthermore, we have the following claims. 

\begin{claim}\label{negativecorr}
For any set $X$ of nonzero elements of $G$, the probability that $X$ is a subset of $S$ is at most $2^{-|X|/2}$.
\end{claim}

This claim follows from first observing that any such $X$ has a subset $X' \subseteq X$ with $|X'|\geq |X|/2$ such that if $x \in X'$, then $-x \not \in X'$. The events that elements of $X'$ taken from different pairs or singletons are in $S$ are independent, while for elements of $X'$ taken from the same pair there is a zero probability that both lie in $S$. Hence, for any such subset $X'$, the probability $X'$ is a subset of $S$ is either $2^{-|X'|}$ or zero.

\begin{claim}\label{fourconsecutive} 
For any nonzero $y \in G$, the quadruple $\{y,2y,4y,8y\} \not \subset S$. 
\end{claim}

This claim follows from the fact that none of the sets of size two in the partition is a subset of $S$. We first observe that we cannot have $2y = -y$ or $8y = -y$, because the order of the group is coprime to $6$. Moreover, if $4y = -y$, then $\{y,2y,4y,8y\} = \{y, 2y, -y, -2y\}$, but $y$ and $2y$ cannot both be in $S$ in this case. We may therefore assume that none of $2y$, $4y$ or $8y$ is $-y$. The quadruple $\{y,2y,4y,8y\}$ might still not consist of distinct elements, but, as $G$ has size coprime to $6$, the only pair in the quadruple that can be equal is $y$ and $8y$. But even in this case, the quadruple (which is then a triple) is not a subset of $S$.  

In particular, for any clique $A$ in the Cayley graph $G_S$ with generating set $S$, since $S$ does not contain any quadruple $\{y,2y,4y,8y\}$, there is no $y\in G$ with $ty\in A-A$ for $t\in \{1,2,4,8\}$. Hence, if $ty\in A-A$ for some nonzero $y \in G$ and all $t=1,2,4,8$, then $A-A$ has nonempty intersection with both $S$ and the complement of $S \cup \{0\}$, so $A$ cannot be a clique in $G_S$. 

The next lemma shows that if $T$ is a set of small positive integers that are coprime to $|G|$ and $A \subseteq G$ is such that there is no nonzero $y\in G$ with $ty\in A-A$ for all $t \in T$, then the ratio $|A-A|/|A|$ is large.

\begin{lem}\label{PRineq}
Let $T$ be a nonempty set of positive integers, each less than $m$, with $|T|=q$ and let $N$ be a positive integer such that every integer in $T$ is coprime to $N$. 
If $A$ is a subset of an abelian group $G$ of order $N$ such that there is no nonzero $y \in G$ with $ty \in A-A$ for all $t \in T$, then $|A-A| \geq |A|^{1+1/(2^q m)}$.  
\end{lem}
\begin{proof}
Order the elements of the set $T$ in increasing order as $t_1<t_2<\dots<t_{q}$ and suppose $|A-A|=K|A|$. 
Beginning with $A_0=A$, we will pick a sequence of subsets $A_0 \supseteq A_1 \supseteq \cdots\supseteq A_q$ with $|A_i| \geq K^{-2^i m + m}|A|$ for $i \geq 1$ and $t_j y\in A-A$ for all $1 \leq j \leq i$ and all $y\in A_i-A_i$. 

Observe that $|A_{i-1}+t_{i}A_{i-1}| \leq |A+t_i A| \le K^{t_{i}+1}|A|$ by the Pl\"unnecke--Ruzsa inequality. 
By the pigeonhole principle, there is $u_i \in A_{i-1}+t_{i}A_{i-1}$ such that $u_i=b+t_ia$ for at least $|A_{i-1}|^2/(K^{t_i+1}|A|)$ pairs $(a,b) \in A_{i-1} \times A_{i-1}$. Let $A_i$ be the set of $a \in A_{i-1}$ such that $u_i - t_i a \in A_{i-1}$, so $|A_{i}| \geq |A_{i-1}|^2/(K^{t_i+1}|A|)$. We will prove by induction on $i$ that $|A_i| \ge K^{-2^i m+m}|A|$ and $t_j y\in A-A$ for all $1 \leq j \leq i$ and all $y\in A_i-A_i$. This holds for $i=0$. Assuming $|A_{i-1}| \ge K^{-2^{i-1}m+m}|A|$, we then have that  
\[
|A_i| \ge |A_{i-1}|^2/(K^{t_i+1}|A|) \ge K^{-(t_i+1)-2^{i}m+2m}|A| \ge K^{-2^{i}m+m}|A|,
\]
where we used that $\max \, T < m$. 
Moreover, for any $a_i,a_i'\in A_i$, we have that $t_i(a_i-a_i')\in A_{i-1}-A_{i-1}\subseteq A-A$. 
By induction, since $A_i \subseteq A_{i-1}$, we also have that $t(a_i-a_i')\in A-A$ for all $t\le t_{i-1}$ in $T$. Thus, we have that $ty\in A-A$ for all $y\in A_i-A_i$ and $t\le t_i$, completing the induction.

If now $|A|>K^{2^q m}$, then $|A_{q}| > 1$ and there is some nonzero $y\in A-A$ with $ty\in A-A$ for all $t\in T$, a contradiction. Therefore, $K\ge |A|^{1/(2^q m)}$ and so $|A-A|=K|A| \ge |A|^{1+1/(2^q m)}$. 
\end{proof}

We are now ready to prove Theorem \ref{thm:coprime-6}. 

\begin{proof}[Proof of Theorem~\ref{thm:coprime-6}]
    For the random $S$ described earlier, if there is a nonzero $y$ such that $\{y,2y,4y,8y\} \subset A-A$, then $A-A$ contains both an element of $S$ and a nonzero element not in $S$, so $A$ is not a clique in the Cayley graph $G_S$ with generating set $S$. If there is no nonzero $y$ such that $\{y,2y,4y,8y\} \subset A-A$, then, for $T=\{1,2,4,8\}$, there is no $y$ with $ty\in A-A$ for all $t\in T$. Therefore, by Lemma \ref{PRineq}, if we set $|A-A|=K|A|$, then we have $K \ge |A|^{1/144}$. 
    
    By Theorem~\ref{thm:count-group}, the number of $A \subset G$ with $|A|=n$ and $|A-A| = m$ is at most $N^{C(K+\log n)}(CK)^n$, where $K=m/n$. By Claim~\ref{negativecorr}, the expected number of cliques in $G_S$ of order $n=C'\log N$ is at most 
    \[
    \sum_{n^{1+1/144} \leq m \leq n^2}N^{C(m/n+\log n)}(Cm/n)^n2^{-(m-1)/2}=o(1).
    \]
    Thus, with high probability, there is no clique in $G_S$ of order $C'\log N$. Similarly, since the complement of the random Cayley graph $G_S$ has the same distribution as $G_S$, with high probability, there is no independent set in $G_S$ of order $C'\log N$. 
\end{proof}

\subsection{More general groups} \label{sec:othergroups}

The main result of this subsection says that if a group $G$ has a large abelian subgroup $H$ of order coprime to $6$, then there is a Cayley graph on the group which is Ramsey. 
As above, the generating set $S$ is not picked uniformly at random (which is necessary for groups with many subgroups of a particular size such as 
$\mathbb{Z}_5^d$). However, 
the elements not in $H$ \emph{are} picked uniformly and independently at random. In particular, if $|H|=o(|G|)$, most elements are picked uniformly and independently at random and we only need to be careful about how we pick the elements 
on a vanishing proportion of the elements of $G$.

\begin{thm}
There is an absolute constant $C'$ such that if $G$ is a group of order $N$ with an abelian subgroup $H$ of order $|H|>N/(\log N)^{1/1000}$ coprime to $6$, then $G$ has a Cayley graph which is $C'$-Ramsey.   
\end{thm}

\begin{proof}
We pick a random subset $S$ of $G$ to be the generating set of the Cayley graph as follows. We pick the nonzero elements of $H$ to be in the generating set exactly as described in the proof of Theorem~\ref{thm:coprime-6} for abelian groups whose size is coprime to $6$. In particular, this implies that every subset $B$ of a right coset of $H$ which could possibly be a clique in the Cayley graph satisfies $|BB^{-1}| \geq |B|^{1+1/144}$. Each element not in $H$ is an element of $S$ with probability $1/2$ independently of all other elements except its inverse. 

Consider a subset $A$ of $n = C' \log N$ elements of $G$. By the pigeonhole principle, there is a right coset of $H$ which contains a subset $B \subset A$ with $|B| \geq n/(N/|H|) \geq n^{.999}$. If $A$ is a clique, then $B$ is as well, so  
$|AA^{-1}| \geq |BB^{-1}| \geq |B|^{1+1/144} \geq n^{1.005}$. 
Thus, by Theorem~\ref{thm:count-group}, 
the number of sets $A \subset G$ of size $n$  with $|AA^{-1}| \leq Kn$ that are potential cliques of the random Cayley graph is at most 
$$N^{C(K+\log n)}(CK)^n = N^{(C+o(1))K},$$
where $C$ is an absolute constant and we used that $K \geq n^{.005}$. 

Note now that each set $A$ with $|A|=n \geq 2$ has probability at most $2^{-(|AA^{-1}|-1)/2}$ of being a clique. Indeed, $AA^{-1}$ is a symmetric set that contains the identity element $1$, each nonidentity element appears in $S$ with probability $1/2$, the events that elements of $AA^{-1} \cap H$ and their inverses appear in $S$ are negatively correlated (as in the proof of Theorem~\ref{thm:coprime-6}) and, finally, each element of $AA^{-1} \setminus H$ appears in the generating set $S$ independently of all other elements except its inverse. Summing over all possible sizes $m=|AA^{-1}|$ of $AA^{-1}$ and substituting in $N=2^{n/C'}$ gives that the expected number of cliques of order $n$ in the random Cayley graph is at most 
$$\sum_{n^{1.005} \leq m \leq n^2} N^{(C+o(1))m/n}2^{-(m-1)/2}=2^{\left((C+o(1))C'^{-1}-1/2\right)m}=o(1).$$ 
Similarly, the expected number of independent sets of order $n$ in this random Cayley graph is $o(1)$. Hence, with high probability this random Cayley graph is $C'$-Ramsey for some appropriate absolute constant $C'$. 
\end{proof}

From the fundamental theorem of finite abelian groups, if $N'$ is the largest factor of $N$ that is coprime to $6$, then any abelian group of order $N$ has a subgroup of order $N'$. This immediately yields Theorem~\ref{main2} and Corollary~\ref{cor3}, with the latter stating that, for almost all $N$, every abelian group of order $N$ has a $C'$-Ramsey Cayley graph.

\subsection{Ramsey Cayley colorings} \label{sec:RamseyCayley}

A Ramsey-type problem first raised by Erd\H{o}s, Hajnal and Rado \cite{EHR} in 1965 (see also~\cite{ErSz72}) asks for the order of the largest clique avoiding one of the colors that can be found in any $r$-coloring of $K_N$. For $r = 2$, this problem is equivalent to the problem of determining the classical diagonal Ramsey number. More generally, for any $1 \leq s < r$, one may ask for the order of the largest clique containing at most $s$ colors that can be found in any $r$-coloring of $K_N$. To get an upper bound, it is easy to check that in a uniform random $r$-coloring the largest clique in the union of any $s$ colors has order $(2+o(1))\log_{r/s} N$ with high probability.

In this subsection, we prove a generalization of Theorem~\ref{thm:Fp} saying that there are Cayley $r$-colorings that asymptotically match this bound, where an $r$-edge-coloring of the complete graph on a group $G$ is a \emph{Cayley $r$-coloring} if each color class is a Cayley graph on $G$. That is, a Cayley $r$-coloring is formed by taking a partition of $G \setminus \{0\}$ into $r$ symmetric sets $S_1,\ldots,S_r$ and coloring the edges of the complete graph on $G$ by letting color $i$ be the Cayley graph with generating set $S_i$ for each $1 \le i \le r$. 

\begin{thm}\label{thm:sharp-clique}
Fix integers $1 \le s < r$. Let $V$ be a finite vector space of order $N$ with characteristic $p>2r$. Then there is a Cayley $r$-coloring of the complete graph on $V$ such that the clique number of the union of any $s$ color classes is at most $(2+o(1))\log_{r/s}{N}$. 
\end{thm}

We also prove Theorem \ref{thm:Fq} in this subsection. 
Recall that this says that for every vector space $V$ of order $N$ whose characteristic $p$ is congruent to $1 \pmod 4$, there is a  self-complementary Cayley graph on $V$ whose clique and independence numbers are both at most $(2+o(1))\log_2 N$. 
We call a Cayley $r$-coloring of $\mathbb{F}_p^n$ \emph{rotational} if the coloring is given by $c: \mathbb{F}_p^{n} \setminus \{0\} \rightarrow \mathbb{Z}_r$ 
and there is a scalar $\alpha \in \mathbb{F}_p^{\times}$ of order $2r$  such that $c(\alpha^i x)\equiv c(x)+i \pmod r$ for all $x\in \mathbb{F}_p^n \setminus \{0\}$ and  $i \in \mathbb{Z}$. In such an edge-coloring of the complete graph on $\mathbb{F}_p^n$, the Cayley graphs of each color are isomorphic (through multiplication by a fixed power of $\alpha$). In particular, in the case $r=2$, the graphs of each color are self-complementary. The following theorem therefore generalizes  Theorem~\ref{thm:Fq}.  

\begin{thm}\label{thm:Fq-2}
    Fix an integer $r \geq 2$. Let $V$ be a finite vector space of order $N$ with characteristic $p \equiv 1\pmod{2r}$. Then there is a rotational Cayley $r$-coloring of the complete graph on $V$ such that the largest monochromatic clique has order at most $(2+o(1))\log_{r}{N}$. 
\end{thm}

We first describe the Cayley $r$-colorings we will use to prove Theorems \ref{thm:sharp-clique} and \ref{thm:Fq-2}. 

\paragraph{Constructions for Theorem \ref{thm:sharp-clique}.}
We use two different colorings, depending on whether the characteristic $p$ is small or not. In each coloring, for each line $L$ through the origin of $\mathbb{F}_p^n$, we color the nonzero elements of $L$ independently of all other lines through the origin. 

If $p \leq 2^{2r}$, then we use the following coloring. For each line $L$ through the origin, we can write its set of $p-1$ nonzero elements as $\{x,2x,\ldots,(p-1)x\}$ for some representative $x \in L$. Since the coloring must be symmetric (so the color of $y$ and $-y$ are the same for each nonzero $y$), we only need to specify the colors of the $(p-1)/2$ elements $x,2x,\ldots,(p-1)x/2$. We equipartition this set of $(p-1)/2$ elements uniformly at random into $r$ sets and then make each set symmetric by adding $-y$ to each set if $y$ is already in the set. We thus obtain a partition of the set of nonzero elements of $L$ into $r$ nonempty sets $L_1,\ldots,L_r$ (the fact that these sets are each nonempty follows from $p>2r$). The generating set $S_i$ for color $i$ is then the union of the sets $L_i$ over all lines $L$ through the origin. In this Cayley $r$-coloring, the union of the generating sets of any set $R$ of colors with $|R|<r$ does not contain all the nonzero elements of a line. Hence, there is no nonzero $y \in \mathbb{F}_p^n$ such that $\{y,2y,\ldots,(p-1)y\}$ are all in $S_R$, where $S_R:=\bigcup_{i \in R} S_i$. But then, by Lemma~\ref{PRineq}, for any $A$ which can possibly be a clique with colors in $R$, we have $|A-A| \geq |A|^{1+p^{-1}2^{-p}}$. 

If $p>2^{2r}$, then we use the following coloring. First note that the order of $2$, which we call $h$, satisfies $h > \log_2 p$. We consider the equivalence relation $x\equiv -x$ which partitions the nonzero elements of $\mathbb{F}_p^n$ into pairs. We write $\overline{x}$ for the equivalence class of $x$. We then partition the nonzero equivalence classes on each line $L$ through the origin into orbits of powers of $2$, each of size $h$ or $h/2$ (depending on whether $-1$ is a power of $2$).

We partition each orbit $B$, which is of the form $\{\overline{x},2\overline{x},2^2\overline{x},\ldots,2^{h'-1}\overline{x}\}$ with $h'=h$ or $h/2$, into intervals $I$ of size $r$ with at most one remaining interval of size less than $r$. Note that there is at least one interval of size $r$ since $h' \geq r$. We color the equivalence classes in each such interval $I$ independently of all other intervals, randomly coloring the elements of $I$ with  colors from $[r]$ while guaranteeing that each color appears at most once (and exactly once if the interval has size $r$). The generating set $S_i$ for color $i$ then consists of all elements of color $i$.   

In this Cayley $r$-coloring, the union of the generating sets of any set $R$ of colors with $|R|<r$ does not contain all the elements of an interval $I$ of size $r$. Hence, there is no nonzero $y$ such that $\{y,2y,\ldots,2^{3r-3}y\}$ are all in $S_R$, where $S_R=\bigcup_{i \in R} S_i$. Here, we note that in the partition of a given orbit into intervals of size $r$ with at most one additional interval of size less than $r$, any consecutive interval of size at least $3r-2$ must fully contain an interval of size $r$ from the partition. But then, by Lemma \ref{PRineq} with $T=\{1,2,\ldots,2^{3r-3}\}$ (so we may pick $q=3r-2$ and $m=2^{3r-3}+1$), for any $A$ which can possibly be a clique with colors in $R$, we have $|A-A| \geq |A|^{1+2^{-6r}}$. 

In either case, for any subset $R$ of the colors of size $s < r$ and any set $X$ of nonzero elements, $X$ has a subset $X'$ of size at least $|X|/2$ which does not contain both an element $x$ and its negative $-x$ and where the events that elements in $X'$ have colors from $R$ are negatively correlated. Therefore, the probability that a set $A \subset \mathbb{F}_p^n$ forms a clique in the union of the Cayley graphs with colors in $R$ is at most $(s/r)^{(|A-A|-1)/2}$. Moreover, as noted above, for $|R| < r$ and any $A$ which has positive probability of being a clique with all colors in $R$, we have 
\begin{equation}\label{eq:doub}
    |A-A| \ge |A|^{1+c}
\end{equation} for some $c>0$ depending only on $r$. 

\paragraph{Constructions for Theorem \ref{thm:Fq-2}.}
We again use two different colorings depending on whether the characteristic $p$ is small or not. As $p \equiv 1\pmod {2r}$, there is $\alpha \in \mathbb{F}_p$ of order $2r$ and so $\alpha^r \equiv -1 \pmod{2r}$. We fix such an $\alpha$. We also let $\ell$ be a multiple of $r$ that is growing slowly to infinity with $N$. For example, $\ell \approx \log \log \log \log N$ will work. The reason that we need to let $\ell \to \infty$ slowly with $N$ is that when the characteristic $p$ is large we lose the negative correlation property that appeared in the construction for Theorem~\ref{thm:sharp-clique}, but we do still have that the probability a set $A$ gives a monochromatic clique is at most $r^{-(1-o(1))|A-A|/2}$, where the $o(1)$ crucially depends on $\ell \to \infty$. However, 
to get our bounds to match those of the uniform random $r$-coloring, we also need that $|A-A| =\omega(|A| \log |A|)$, which we can guarantee with our arguments only if $\ell$ doesn't grow too quickly with $N$.

When $p < 2^{2r\ell}$, we use the following coloring. The set $\mathbb{F}_p^n \setminus \{0\}$ of nonzero elements is partitioned into the nonzero elements of the lines through the origin. 
We can write the nonzero elements of a line through the origin as $L_x=\{ax:a \in \mathbb{F}_p^{\times}\}$ for some $x \in \mathbb{F}_p^n \setminus \{0\}$. 
We color the nonzero elements of each such line independently of the other lines through the origin. 
Pick a generator $g$ for $\mathbb{F}_p^{\times}$. 
We equipartition the set of elements $g^j x$ for $j=0,1,\ldots,\frac{p-1}{2r}-1$ into $r$ color classes uniformly at random and then complete this to a rotational coloring of the line by letting $c(\alpha^i g^jx) \equiv c(g^jx)+i \pmod r$ for $j=0,1,\ldots,\frac{p-1}{2r}-1$ and $i \in \mathbb{Z}$. To ensure that the coloring is well-defined, we claim that the sets $\{\pm \alpha^i g^j\}$ for $j\in [0,\frac{p-1}{2r}-1]$ and $i\in [0,r-1]$ partition $\mathbb{F}_p \setminus \{0\}$. 

Assume instead that the sets $\{\pm \alpha^i g^j\}$ for $j\in [0,\frac{p-1}{2r}-1]$ and $i\in [0,r-1]$ are not disjoint, so there is $|j'| < \frac{p-1}{2r}$ and $|i'| < r$, not both zero, such that $\alpha^{i'} = \pm g^{j'} \pmod p$. We then have $g^{rj'} = (\pm \alpha^{i'})^r = \pm 1 \pmod p$ and, since $|rj'| < (p-1)/2$, this implies $rj'=0$ and hence $j'=0$. Then $\alpha^{i'} = \pm 1\pmod p$ and, since $\alpha$ has order $2r$ and $|i'|<r$, we have $i'=0$, a contradiction. Therefore, the sets $\{\pm \alpha^i g^j\}$ for $j\in [0,\frac{p-1}{2r}-1]$ and $i\in [0,r-1]$ are disjoint and, as there are $(p-1)/2$ such sets, they must also partition $\mathbb{F}_p \setminus \{0\}$.
Since the nonzero elements of a line through the origin cannot be monochromatic, Lemma~\ref{PRineq} again implies that, for any $A$ which can possibly be a monochromatic clique, we have $|A-A| \geq |A|^{1+p^{-1}2^{-p}}$. 

When $p \geq 2^{2r\ell}$, we use the following coloring. Partition $\mathbb{F}_p^n \setminus \{0\}$ into sets of the form $M_x=\{2^a \alpha^i x:a,i \in \mathbb{Z}\}$ for some representative $x \in \mathbb{F}_p^n \setminus \{0\}$. We will color the elements of each $M_x$ independently of the elements not in $M_x$. 
Let $h$ be the order of $2$ in $\mathbb{F}_p^{\times}$, so $h>\log_2 p$. Let $h'$ be the least positive integer such that $2^{h'}$ is a power of $\alpha$ in $\mathbb{F}_p^{\times}$. Then $h'=h/d$ for some divisor $d$ of $h$ with $d \leq 2r$ and each $M_x$ has size $2rh'$. Partition the set $\{0,1,\ldots,h'-1\}$ into intervals of size $\ell$, with possibly one interval of size at least $\ell$ and at most $2\ell-1$. We can do this since $h' \geq \ell$. For each such interval $I$, we equitably color the elements of $\{2^ix:i \in I \}$ at random. We then extend this to a rotational coloring of all of $M_x$ by setting $c(\alpha^i y)\equiv c(y)+i \pmod{r}$. For any monochromatic clique $A$ in this edge-coloring, the set $S=\{2^j\alpha^i x:j\in I\}$ for a given $i$, $x \in \mathbb{F}_p^n \setminus \{0\}$ and interval $I$ cannot be monochromatic and, hence, cannot be a subset of $A-A$.  Further, for every nonzero $y\in \mathbb{F}_p^n$, the set $\{y,2y,\ldots,2^{3\ell-3}y\}$ contains such a set $S$. Hence, by Lemma~\ref{PRineq}, for any $A$ which can possibly be a monochromatic clique, we have $|A-A| \geq |A|^{1+\ell^{-1}2^{-3\ell}}$. 

Recall that, for any $X \subset \mathbb{F}_p^n \setminus \{0\}$, there is $X' \subset X$ with $|X'| \geq |X|/2$ such that there is no $x$ with both $x$ and $-x$ in $X'$. For the coloring used when $p<2^{2r\ell}$, for any set of the form $X=A-A$, the events that elements of the corresponding $X'$ each have the same particular color are negatively correlated and each have probability $1/r$. Therefore, the probability that a set $A \subset \mathbb{F}_p^n$ forms a monochromatic clique of a given color in the rotational Cayley $r$-coloring is at most $r^{-(|A-A|-1)/2}$. 

For the coloring used when $p \geq 2^{2r\ell}$, 
we no longer have negative correlation. However, the events are still close to being negatively correlated, so we can show that the probability a set $A$ forms a monochromatic clique of a given color 
is at most $r^{-(1+o(1))|A-A|/2}$. 
For a set $A$ to have positive probability of being a monochromatic clique of color $c \in \mathbb{Z}_{r}$, 
there must be no $y \in X'$ and $\alpha^i y \in X'$ with $i \not \equiv 0 \pmod{2r}$ of the same color, where $X'$ is the reduced set corresponding to $X=A-A$. For a given $M_x$ and an interval $I$ as defined above (which satisfies 
$\ell \leq |I| \leq 2\ell-1$), 
each $y \in M_x \cap X'$ of the form $y=\alpha^j 2^i x$ for $i \in I$ receives color $c$ if and only if $2^i x$ receives the color $c-j$. 
Therefore, for a set $A$ to have positive probability of forming a monochromatic clique, several independent events must occur, 
each saying that $2^i x$ receives a certain color for each $i \in J \subset I$, where the subsets $J=J(A,x,I)$ have sizes that sum to $|X'|$. The probability that 
each $i \in J$ receives the required color is the ratio of two multinomial coefficients, the denominator being the number of equitable $r$-colorings of $|I|$, which is $${|I| \choose n_1,~n_2,~\ldots,~n_r}$$ with each $n_k=|I|/r$, and the numerator being 
$${|I|-|J| \choose n_1-m_1,~n_2-m_2,~\ldots,~n_r-m_r},$$ 
where $m_k$ is the number of elements of the form $2^i x$ with $i \in J$ that are to receive color $k$, so $m_1+\cdots+m_r=|J|$. The ratio of these two multinomial coefficients is then 
$$Q_I:=\frac{\prod_{k=1}^r (|I|/r)_{m_k}}{(|I|)_{|J|}},$$ where $(a)_{b}=\prod_{j=1}^{b} (a-j+1)$. 
For any $A$ that has positive probability of being monochromatic,  
the probability that $X'$ is monochromatic is $\prod_I Q_I$. We wish to compare this with $r^{-|X'|}$. 
To this end, we will give an upper bound on $Q_I^{1/|J|}$. By the simple observation that, for any integers $Y \geq a > b \geq 0$, we have $(Y)_{a}(Y)_{b} \leq (Y)_{a-1}(Y)_{b+1}$, given $|J|$, the ratio $Q_I$ is maximized when $(m_1,\dots,m_r)$ is an equipartition of $|J|$. Furthermore, the maximum of $Q_I^{1/|J|}$ occurs when $|J|=|I|=\ell$, so 
$$Q_I^{1/|J|} \leq {\ell \choose m_1,~m_2,~\ldots,~m_r}^{-1/\ell},$$
with each $m_k=\ell/r$. This multinomial coefficient is the largest of the $\binom{\ell+r-1}{r-1} \leq \ell^r$ multinomial coefficients in the multinomial expansion of $r^{\ell}$, so, since $\ell \to \infty$, we have $Q_I \leq r^{-(1-o(1))|J|}$. 
Taking the product over all $I$, we get that the probability a set $A$ (with $|A| \to \infty$ as $N \to \infty$) forms a monochromatic clique is, as required, at most $r^{-(1-o(1))|A-A|/2}$. 

Moreover, as $\ell$ is tending to $\infty$ slowly with $N$, in both cases any $A$ of size $(2+o(1)\log_r n$ which has a positive probability of being a monochromatic clique satisfies 
\begin{equation*}\label{eq:doub-2}
    |A-A| \ge |A|(\log |A|)^2. 
\end{equation*}

\vspace{3mm}

\noindent {\bf Proofs of Theorems \ref{thm:sharp-clique} and \ref{thm:Fq-2}.}
The proofs of Theorems \ref{thm:sharp-clique} and \ref{thm:Fq-2} use the colorings described above together with the following upper bound on the number of sets of a given size with small difference set. The crucial point for us here is that the constant factor in front of $K$ in the exponent of $N$ is exactly $1$. 

\begin{lem}\label{lem:count-K}
    For $K\ge 1$, the number of subsets $A$ of $\mathbb{F}_p^n$ of size $t$ with $|A-A|\le Kt$ is at most $t^{4t}N^{K+\log_p t+1}$, where $N=p^n$. 
\end{lem}

This bound is close to optimal. More precisely, for $K$ a positive integer and $t/K$ a power of $p$, the number of subsets $A$ of $\mathbb{F}_p^n$ of size $t$ with $|A-A|\le Kt$ is at least $(N/t)^{K+\log_p (t/K)}$. 
Indeed, consider those sets $A = H + \{x_1,\dots,x_K\}$ which are the union of $K$ distinct cosets of a subspace $H$ of size $t/K$. Any such set $A$ satisfies $|A-A| \leq (K^2-K+1)|H| \leq K|A|$. As we assumed $t/K$ is a power of $p$, the number of distinct subspaces of size $t/K$ is 
\[
    \prod_{j=0}^{\log_p(t/K)-1} \frac{N-p^j}{t/K-p^j} \ge \left(\frac{N}{t/K}\right)^{\log_p(t/K)}.
\]
Thus, the number of such $A = H + \{x_1,\dots,x_K\}$ is at least 
\[
    \binom{N/(t/K)}{K}\left(\frac{N}{t/K}\right)^{\log_p(t/K)} \ge 
    (N/t)^{K + \log_p(t/K)},
\]
where we note that there are $\binom{N/(t/K)}{K}$ ways to choose $K$ distinct $H$-cosets. 
We defer the proof of Lemma~\ref{lem:count-K} to the next subsection, first showing how to complete the proofs of Theorems \ref{thm:sharp-clique} and \ref{thm:Fq-2}.

\begin{proof}[Proof of Theorem \ref{thm:sharp-clique}]
Fix for now a nonempty proper subset $R$ of the set of colors with $|R|=s$. As discussed earlier, the probability that a set $A\subset \mathbb{F}_p^n$ forms a clique with all colors in $R$ is at most $(s/r)^{(|A-A|-1)/2}$. Moreover, by Lemma \ref{lem:count-K}, the number of sets $A$ with $|A|=t$ and $|A-A|\le m$ is at most $t^{4t} N^{m/t+\log_p t+1}$ and, by (\ref{eq:doub}), any clique $A$  with all colors in $R$ has $|A-A|\ge |A|^{1+c}$ for some $c>0$ depending only on $r$. 

Summing over all possible sizes $m$ of $|A-A|$ with $m\ge |A|^{1+c}$, the probability that there exists a set $A\subset \mathbb{F}_p^n$ with $|A|=t$ which forms a clique with all colors in $R$ is at most 
\[
\sum_{t^{1+c}\le m\le t^2} t^{4 t}N^{m/t+\log_p t + 1} (s/r)^{(m-1)/2}.
\]
Let $\delta=t^{-c/2}$ and $t=2(1+\delta) \log_{r/s}(N)$. Then,
\begin{align*}
    \sum_{t^{1+c}\le m\le t^2} t^{4t}N^{m/t+\log_p t+1} (s/r)^{(m-1)/2} &\le  \sum_{t^{1+c}\le m\le t^2} t^{4t}N^{m/t+\log_p t+1} N^{-(1+\delta)m/t}r^{1/2} \\
    &= \sum_{t^{1+c}\le m\le t^2}t^{4t} N^{\log_p t + 1 - \delta m/t}r^{1/2}\\
    &\le N^{-t^{c/2}/2} t^{4t+2}r^{1/2} \\
    &=(s/r)^{t^{1+c/2}/4(1+\delta)} t^{4t+2}r^{1/2}\\
    &= o_N(1). 
\end{align*}
Taking a union bound over all $2^r-2$ possible nonempty proper subsets $R$ of the set of colors, we get that, with high probability, the clique number of the union of any $s<r$ color classes is at most $(2+o_N(1))\log_{r/s}(N)$. 
\end{proof}

\begin{proof}[Proof of Theorem \ref{thm:Fq-2}]
    The proof goes through essentially the same as the proof of Theorem \ref{thm:sharp-clique} with 
    the analogous estimates substituted.  
\end{proof}

\subsection{Counting sets with small difference sets in vector spaces}\label{subsec:EZL}

We now prove Lemma~\ref{lem:count-K}, for which we will need a quantitative version of the Freiman--Ruzsa theorem. One version of this theorem, due to Ruzsa~\cite{R99}, says that if $G$ is a finite abelian group of exponent $t$, then any subset $A$ of $G$ with $|A + A| \leq  K|A|$ is contained in a coset of a subgroup of $G$ of size at most $K^2 t^{K^4}|A|$.  In the same paper, Ruzsa conjectured that the subgroup size can be reduced to $t^{CK}|A|$ for some absolute constant $C$ and this conjecture was subsequently proved by Even-Zohar and Lovett~\cite{EZo, EZL} for prime exponents. 

Let $\langle A \rangle$ denote the \emph{affine span} of $A$, that is, the smallest coset of a subgroup containing $A$. If we write 
$$F(p,K)=\sup_{A}\left\{\frac{| \langle A \rangle|}{|A|} \, | \, A\subseteq \mathbb{F}_p^n \text{ for some $n$}, |A+A|\leq K |A| \right\},$$ then the result of Even-Zohar and Lovett~\cite{EZL} says that $F(p,K) \leq \frac{p^{2K-2}}{2K-1}$ for $p>2$ prime and $K$ sufficiently large.  
Here we adapt their proof to the setting where we replace sumsets by difference sets. If we write
$$F^{-}(p,K)=\sup_{A}\left\{\frac{| \langle A \rangle|}{|A|} \, | \, A\subseteq \mathbb{F}_p^n\text{ for some $n$}, |A-A|\leq K |A| \right\},$$ 
then the following is the result we need.

\begin{thm}\label{thm:EZL} For $K\ge 1$ and any prime $p\ge 3$, 
\begin{equation*}
    F^{-}(p,K)\leq \frac{p^{K}}{K+2}.
\end{equation*}
\end{thm}

By considering $A = \langle e_1,\dots,e_\ell\rangle + \{0,e_{\ell+1},\dots,e_{\ell+K}\}$, we have that $|A-A| = (1 + 2K + K(K-1))p^{\ell} = (K + 1/(K+1)) |A|$, while $|\langle A\rangle| = \frac{p^{K}}{K+1} |A|$. Thus, Theorem \ref{thm:EZL} is tight up to a $p^{O(1/K)}$ factor. 

Our proof of Theorem \ref{thm:EZL} follows~\cite{EZL} quite closely. Without loss of generality, in the definition of $F^-(p,K)$, by passing to a subspace if necessary, we may restrict the supremum to subsets $A$ of $\mathbb{F}_p^n$ which affinely span $\mathbb{F}_p^n$. 
By changing basis if necessary, we may also assume that $A$ contains the standard basis $e_1,\dots,e_n$ of $\mathbb{F}_p^n$.  

We consider the \emph{lexicographic order} on the elements of $\mathbb{F}_p^n$. Recall that, for two vectors $u,v\in \mathbb{F}_p^n$, we write $u\prec v$ in this ordering if $u_i<v_i$ for the largest coordinate $i$ for which $u_i\neq v_i$, where the coordinates are viewed as integers in $\{0,1,\ldots,p-1\}$. For $L\subseteq \mathbb{F}_p^n$ and $k\leq |L|$, let $\is(k, L)$ denote the \emph{initial segment} of size $k$ of $L$, which is the set of the $k$ smallest elements of $L$ in the lexicographic order. 

For a given $v\in \mathbb{F}_p^n$, let $\mathcal{L}_v$ be the set of all lines in the direction $v$, that is, $\mathcal{L}_v=\bigcup_{u\in \mathbb{F}_p^n}{L_v^u}$, where $L_v^u=\{u+ t v|t\in \mathbb{F}_p\}$. 
Given a set $A$, we define its \emph{compression in the direction} $v$, denoted by $C_v(A)$, to be the set obtained from $A$ by replacing each set $L\cap A$, for all $L=L_v^u$, $u\in \mathbb{F}_p^n$, with the initial segment of $L_v^u$ of the same cardinality, that is, $$C_v(A)=\bigcup_{u\in \mathbb{F}_p^n}{\is(|L_v^u\cap A|, L_v^u)}.$$

For a vector $v\in \mathbb{F}_p^n$, we denote by $\ell(v)$ the largest index $i$ for which $v_i\ne 0$. The following key lemma is the analogue of a similar result in~\cite{EZL}.

\begin{lem} If $A,B\subseteq \mathbb{F}_p^n$ and $v \in \mathbb{F}_p^n$ is a vector with $v_{\ell(v)}=1$, then $|C_v(A)-C_v(B)|\le |C_v(A-B)|$.
\end{lem}

\begin{proof}
    Let $i=\ell(v)$. For $u,x\in \mathbb{F}_p^n$, let $u_0,x_0$ be the translates of $u,x$ by scalar multiples of $v$ such that their $i$-th coordinates are $0$. Note that $\textrm{IS}(|A\cap L_v^u|,L_v^u) = \{u_0,\dots,u_0+(|A\cap L_v^u|-1)v\}$ as $v_i = 1$. Hence, 
    \[|(C_v(A)-C_v(B))\cap \{x_0+tv | t\in \mathbb{F}_p\}| = \min(p,\max_{u_0}(|A\cap \{u_0+tv|t\in \mathbb{F}_p\}|+|B\cap \{x_0-u_0+tv | t\in \mathbb{F}_p\}|)-1).\]
    On the other hand, 
    \begin{align*} 
        |C_v(A-B) \cap \{x_0+tv | t\in \mathbb{F}_p\}| &= |(A-B)\cap \{x_0+tv | t\in \mathbb{F}_p\}| \\
        &\ge \min(p,\max_{u_0}(|A\cap \{u_0+tv|t\in\mathbb{F}_p\}|+|B\cap \{x_0-u_0+tv | t\in \mathbb{F}_p\}|)-1). 
    \end{align*}
    Thus, $|C_v(A)-C_v(B)|\le |C_v(A-B)|$, as required.
\end{proof}

In \cite{EZL}, the authors use compressions along directions $v$ with $v_{\ell(v)}=1$ to deduce the structure of compressed sets. We say that $A$ is *-compressed if $A=C_v(A)$ for any $v$ with $v_{\ell(v)}=1$ and $e_{\ell(v)} - v \in A$. The following lemma follows identically to Lemma 7 of \cite{EZL}. Recall that $e_1,\dots,e_n$ denotes the standard basis vectors of $\mathbb{F}_p^n$, which we assume are contained in $A$. We also observe that given any set $A$ containing the standard basis, any compression in the direction of $v$ where $v_{\ell(v)}=1$ and $e_{\ell(v)}-v\in A$ would fix the standard basis in $A$, so $A$ remains affinely spanning.

\begin{lem}\label{lem:compressed}
    Let $A$ be a subset of $\mathbb{F}_p^n$ with $\{e_1,\dots,e_n\}\subseteq A$ which is *-compressed. Let $h$ be maximal such that $H=\langle e_1,\dots,e_h\rangle$ is a subset of $A$. Then $A$ can be written as the disjoint union 
    \[
    H \cup (e_{h+1}+H) \cup \dots \cup ((s-1)e_{h+1}+H) \cup (A\cap (se_{h+1}+H)) \cup (A\cap (e_{h+2}+H)) \cup \dots \cup (A\cap (e_n+H)),
    \]
    where $1\le s\le p-1$ and $s$ is maximum with the property that $(A\cap (se_{h+1}+H)) \neq \emptyset$. 
\end{lem}
We are now ready to prove Theorem \ref{thm:EZL}.
\begin{proof}[Proof of Theorem \ref{thm:EZL}]
Let $A\subseteq \mathbb{F}_p^n$ be such that $\{e_1,\dots,e_n\}\subseteq A$, so $\langle A\rangle = \mathbb{F}_p^n$. We repeatedly apply compressions along directions $v$ satisfying $v_{\ell(v)}=1$ and $e_{\ell(v)}-v\in A$. For each such compression, the difference set $|A-A|$ does not increase and the standard basis $e_1,\dots,e_n$ remains in $A$. By applying finitely many such compressions, we arrive at a set $A\subseteq \mathbb{F}_p^n$ with $|A-A| \le K|A|$ and $\{e_1,\dots,e_n\} \subseteq A$ which is *-compressed. $A$ must then have the form given in Lemma \ref{lem:compressed}, so that 
\begin{align*}
A-A &= [H + ([-s,s]e_{h+1} \cup (\{e_{h+2},\dots,e_n\} - [0,s-1]e_{h+1}) \cup (-\{e_{h+2},\dots,e_n\} + [0,s-1]e_{h+1}))] \cup [X-X],
\end{align*}
where $X=(A\cap (se_{h+1}+H)) \cup (A\cap (e_{h+2}+H)) \cup \dots \cup (A\cap (e_n+H))$. 
 
We have 
\[
| H + ([-s,s]e_{h+1} \cup (\{e_{h+2},\dots,e_n\} - [0,s-1]e_{h+1}) \cup (-\{e_{h+2},\dots,e_n\} + [0,s-1]e_{h+1})) | = |H|(2s+1 + 2 s (n-h-1)). 
\]
Furthermore, if we set $A_j = A\cap (e_{j}+H)$ for $j\ge h+2$ and $A_{h+1} = A\cap (se_{h+1}+H)$, then, for $j\ne j' \in [h+1,n]$, the sets $A_{j'}-A_j$ are disjoint from each other and from the set $$H + ([-s,s]e_{h+1} \cup (\{e_{h+2},\dots,e_n\} - [0,s-1]e_{h+1}) \cup (-\{e_{h+2},\dots,e_n\} + [0,s-1]e_{h+1})).$$ 
Thus, 
\[
|A-A| \ge (2s(n-h)+1)|H| + \sum_{j\ne j', \, j,j'\ge h+1} |A_j-A_{j'}| \ge (2s(n-h)+1)|H| + \sum_{j\ge h+1} (n-h-1)|A_j|, 
\]
where, in the last inequality, we used that $\sum_{j'\ne j} |A_j-A_{j'}| \ge (n-h-1)|A_j|$. 

On the other hand, 
\[
|A| = s|H|+\sum_{j\ge h+1}|A_j|. 
\]
Hence, 
\[
|A-A|\ge (2s(n-h)+1)|H| + (n-h-1)(|A|-s|H|) = (n-h-1)|A| + (1 + s(n-h+1))|H|. 
\]
Thus,
\begin{equation}\label{eq:K}
K \ge |A-A|/|A| \ge n-h-1+(1+s(n-h+1))p^h/|A|,
\end{equation}
so
\begin{equation}\label{eq:upper-A}
|A|\ge \frac{(1+s(n-h+1))p^h}{K-n+h+1}. 
\end{equation}

From (\ref{eq:upper-A}), recalling that $\langle A\rangle = \mathbb{F}_p^n$ and noting that $1+s(n-h+1)\ge n-h+2$, we have 
\begin{align*}
\frac{|\langle A\rangle|}{|A|} &= \frac{p^n}{|A|} \le \frac{(K-n+h+1)p^{n-h}}{1+s(n-h+1)} \le \frac{(K-n+h+1)p^{n-h}}{n-h+2} 
= p^K \cdot \frac{K-n+h+1}{p^{K-n+h}(K-(K-n+h-2))}.
\end{align*}
Since $A\subseteq H\cup (e_{h+1}+H)\cup \dots \cup (se_{h+1}+H)\cup (e_{h+2}+H)\cup \dots \cup (e_n+H)$, we have $|A| \le (n-h+s)p^h$. Thus, from (\ref{eq:K}), 
\[
K \ge n-h-1+\frac{(1+s(n-h+1))p^h}{|A|} \ge n-h-1 + \frac{1+s(n-h+1)}{n-h+s} \ge n-h. 
\]
Let $z=K-n+h\ge 0$. Recalling that $p\ge 3$, we then have that 
\[
\frac{z+1}{p^z (K-z+2)} \le \frac{1}{K+2}, 
\]
so that 
\begin{equation*}
    \frac{|\langle A\rangle |}{|A|}\le \frac{p^{K}}{K+2},
\end{equation*}
as required. 
\end{proof}

We now recall some basic definitions from additive combinatorics.  
For subsets $A\subseteq G$ and $B\subseteq H$ of abelian groups $G,H$, a mapping $\varphi:A\rightarrow B$ is a \emph{Freiman $(2-)$homomorphism} if $\varphi(a_1)+\varphi(a_2)=\varphi(a_1')+\varphi(a_2')$ for any $a_1, a_2, a_1', a_2'\in A$ with $a_1 +a_2=a_1' +a_2'$. If, in addition,  $\varphi$ is bijective and its inverse is also a Freiman homomorphism, then $\varphi$ is called a \emph{Freiman isomorphism}. The \emph{Freiman dimension} of a subset $A$ of $\mathbb{F}_p^n$ is the largest $d$ such that there exists a subset of $\mathbb{F}_p^d$ with full affine span that is Freiman isomorphic to $A$. As a corollary of Theorem \ref{thm:EZL}, we obtain that the Freiman dimension over $\mathbb{F}_p$ of a set $A\subseteq \mathbb{F}_p^n$ with $|A-A|\le K|A|$ is at most $K+\log_p |A|$. We now use this corollary to prove Lemma \ref{lem:count-K}, which says that the number of subsets $A$ of $\mathbb{F}_p^n$ of size $t$ with $|A-A|\le Kt$ is at most $t^{4t}N^{K+\log_p t+1}$, where $N = p^n$. 

\begin{proof}[Proof of Lemma \ref{lem:count-K}]
The proof follows the proof of Proposition 26 of \cite{G05}. Lemma 11 of \cite{G05} implies that the number of Freiman isomorphism classes of sets in $\mathbb{F}_p^n$ of size $t$ is at most $t^{4t}$. Lemma 13(i) of \cite{G05} says that if $A$ has Freiman dimension $d$, then there is a subset $\overline{A}$ of $A$ of size $d+1$ such that any Freiman isomorphism between $A$ and a subset $B$ of $\mathbb{F}_p^n$ is uniquely determined by the image of $\overline{A}$. Thus, for each Freiman isomorphism class consisting of sets with Freiman dimension $d$, the number of subsets of $\mathbb{F}_p^n$ in that isomorphism class is at most $N^{d+1}$. 
Therefore, using our bound on the Freiman dimension of $A$, the number of $A\subseteq \mathbb{F}_p^n$ of size $t$ with doubling at most $K$ is at most $t^{4t} N^{K+\log_p t + 1}$, as required. 
\end{proof}

\section{Concluding remarks}

\noindent {\bf Other random graph models.} A different model of random graphs extending random Cayley graphs was proposed and studied by Christofides and Markstr\"om \cite{CM}. Recall that a \emph{Latin square} is an $N \times N$ square filled with symbols from $[N]$ in such a way that each symbol appears exactly once in each row and column. Given an $N \times N$ Latin square $L$ and a subset $S$ of $[N]$, the \emph{Latin square graph} $G(L,S)$ is the graph on vertex set $[N]$ where $ij$ is an edge if and only if either $L_{ij}$ or $L_{ji}$ is in $S$. Christofides and Markstr\"om initiated the study of the \emph{random Latin square graph} $G(L, p)$, which, for a given $N \times N$ Latin square $L$, is the Latin square graph $G(L, S)$ where each element of $[N]$ is included in the generating set $S$ independently with probability $p$. This clearly generalizes random Cayley graphs, since one can just start with the multiplication table of a given finite group $G$ as the relevant Latin square.

Among other results, Christofides and Markstr\"om \cite{CM} proved that both the clique and independence numbers of random Latin square graphs are $O(\log^2 N)$ with high probability. They also suggested that both of these bounds might be improvable to $O(\log N \log \log N)$, which would be best possible for the same reason as for random Cayley graphs. 
Since a random Latin square graph can be viewed as the edge union of two random entangled graphs, it easily follows from our results that random Latin square graphs have independence number $O(\log N \log \log N)$ with high probability, verifying one part of their suggestion. However, for the clique number, the corresponding bound does not directly follow.

Nevertheless, all of the random graph models we have discussed, including binomial random graphs, random Cayley graphs, random entangled graphs coming from locally-bounded colorings and random Latin square graphs, fall into a more general framework where each edge $e$ in a complete graph appears with a fixed probability $0<p<1$ independently of all edges except for those in a bounded-degree subgraph $G_e$. In a forthcoming paper~\cite{CFPY+}, we will study such random graphs and, among other results, extend Theorem~\ref{main3}, our result estimating the clique number of random entangled graphs, to this setting. In particular, this confirms that the clique number of random Latin square graphs is $O(\log N \log \log N)$ with high probability, confirming the second part of Christofides and Markstr\"om's suggestion.

\vspace{3mm}

\noindent {\bf Uniform edge distribution and extractors.}
Building on the techniques in this paper, it is possible to show that the edges of the random Cayley graph over any finite group $G$ are uniformly distributed with high probability in the sense that there is an absolute constant $c > 0$ such that the density of edges between any two sets of order at least $\log |G|(\log \log |G|)^{c}$ is $\frac 12 + o(1)$ as $|G|$ tends to infinity. This improves upon a result of Konyagin and Shkredov~\cite{KS}, who proved a similar result for abelian groups.

These results also connect to the study of randomness extractors in theoretical computer science. Following Chattopadhyay and Liao~\cite{CL}, who studied the abelian case under the name of sumset extractors, we say that $f:G\to \{0,1\}$ is a \emph{$(G, k,\epsilon)$-extractor} if, for all distributions $P_1, P_2$ on a finite group $G$ with $\max_{g\in G}P_i(g) \le 2^{-k}$ and independent samples $U_i \sim P_i$, we have that $d_{\mathrm{TV}}(f(U_1U_2^{-1}), \mathrm{Ber}(1/2)) \le \epsilon$, where $d_{\mathrm{TV}}$ is the total variation distance. A standard reduction shows that it suffices to verify this property when each $U_i$ is the uniform distribution on a subset of $G$ of order at least $2^k$. That is, we only need to show that the edges of the Cayley graph defined by $f$ are close to uniformly distributed between all pairs of sets each of order at least $2^k$, which is exactly the problem we have discussed above.

In their paper~\cite{CL}, Chattopadhyay and Liao constructed explicit sumset extractors over $\mathbb{F}_2^n$ for $k = \log n (\log \log n)^{1+o(1)}$ and asked whether random functions $f$ give good extractors. The results of Konyagin and Shkredov~\cite{KS} address this question, showing that in any finite abelian group uniformly random functions are good extractors for $2^k = \log |G|(\log \log |G|)^{c}$ and $|G|$ sufficiently large. Our results show that this remains true for arbitrary finite groups. 

\vspace{3mm}

\noindent {\bf Abelian Ramsey Cayley graphs in small characteristic.} 
Alon's conjecture remains wide open for $\mathbb{F}_2^n$ and $\mathbb{F}_3^n$. The main challenge in these cases is that it is no longer possible to directly construct a $2$-coloring of the nonzero elements of these groups so that there is no monochromatic subspace. In fact, a classical result in Ramsey theory, the finite unions theorem, shows that this is not possible over $\mathbb{F}_2^n$ for $n$ sufficiently large. As such, to make progress on Alon's conjecture for $\mathbb{F}_2^n$, we need to better understand the quantitative bounds for Ramsey numbers of subspaces in $\mathbb{F}_2^n$. In particular, we propose the following conjecture, which would be an automatic corollary of Alon's conjecture in $\mathbb{F}_2^n$ if it were true, as an intermediate step.

\begin{conjecture}\label{subspaceconjecture}
    There exists $C>0$ and a two-coloring of $\mathbb{F}_2^{n}\setminus \{0\}$ such that there is no subspace $H$ of size $Cn$ for which all nonzero elements of the subspace are monochromatic. 
\end{conjecture}

While this would be an important step towards proving Alon's conjecture for $\mathbb{F}_2^n$, it may be that the true bound is considerably smaller. The best lower bound we know of follows from an appropriate adaptation of Sanders' quantitative bounds for Rado's theorem~\cite{S} to the finite field setting and gives that every two-coloring of $\mathbb{F}_2^{n}\setminus \{0\}$ contains a monochromatic subspace of size at least $\log \log \cdots \log n$ for some bounded number of iterated logs.

\end{document}